\documentclass[a4paper,12pt]{article}
\usepackage{array,amsthm,amsfonts,epsfig}
\usepackage{graphicx}
\usepackage{dina42e}
\usepackage{amsmath}
\usepackage{mathdots}
\usepackage{amssymb}
\usepackage{bbm}
\usepackage{textcomp}
\usepackage[T1]{fontenc}
\usepackage{fancyhdr}
\usepackage[utf8]{inputenc} 
\usepackage{euscript} 
\usepackage{mathrsfs}  
\usepackage{calligra} 
\usepackage{xcolor}

\newtheorem{satz}{Theorem}[section]
\newtheorem{thm}[satz]{Theorem}
\newtheorem{lemma}[satz]{Lemma}
\newtheorem{kor}[satz]{Corollary}

\theoremstyle{definition}
\newtheorem{Def}[satz]{Definition}
\newtheorem{defn}[satz]{Definition}
\theoremstyle{remark}
\newtheorem{bem}[satz]{Remark}

\newcommand{\dq}{\textrm{\textit{\dj}} }                                           
\newcommand{\<}[1]{\langle #1 \rangle}                                    
\newcommand{\op}{OP}                                             
\newcommand{\e}{\varepsilon}                                              
\newcommand{\id}{\textrm{Id}}                                             
\newcommand{\supp}{\textrm{supp }}                                        
\newcommand{\codim}{\textrm{codim }}                                      
\renewcommand{\Im}{\textrm{Im}}                                            
\newcommand{\Span}{\textrm{span }}                                        
\newcommand{\ind}{\textrm{ind}}                                          
\newcommand{\Sallg}[5]{S^{#1}_{#2,#3}(\R^{#4}\times\R^{#5})}              
\newcommand{\Sn}[3]{\Sallg{#1}{#2}{#3}{n}{n}}                             
\newcommand{\pa}[1]{\partial_{\xi}^{#1}}                                  
\newcommand{\p}{\partial}
\newcommand{\pax}[2]{\partial_{\xi}^{#1}\partial_{x}^{#2}}                
\newcommand{\skh}[3]{{\langle #1,#2 \rangle}_{#3}}                        
\newcommand{\R}{\mathbb{R}}                                               
\newcommand{\Rn}{\mathbb{R}^n}                                            
\newcommand{\RnRn}{\R^n \times \R^n}                                      
\newcommand{\RnRnx}[2]{\R^n_{#1} \times \R^n_{#2}}                         
\newcommand{\intr}{\int \limits_{\Rn} }                                   
\newcommand{\osint}{\textrm{Os\hspace*{0,1cm}-}\hspace{-0,15cm}\iint}     
\newcommand{\s}{\mathcal{S}(\R^n)}                                        
\newcommand{\sd}{\mathcal{S'}(\R^n)}                                      
\newcommand{\N}{\mathbb{N}}                                               
\newcommand{\Z}{\mathbb{Z}}                                               
\newcommand{\Non}{\mathbb{N}_0^n}                                               
\newcommand{\C}{\mathbb{C}}                                               

\title{Invariance of the Fredholm Index and Spectrum of Non-Smooth Pseudodifferential Operators}
\author{Helmut Abels\footnote{Fakult\"at f\"ur Mathematik,  
Universit\"at Regensburg,
93040 Regensburg,
Germany, e-mail: {\sf helmut.abels@ur.de}}\ \ and Christine Pfeuffer\footnote{Institute f\"ur Mathematik,  
Martin-Luther-Universität Halle-Wittenberg,
06099 Halle (Saale),
Germany}}
\date{}
\begin{document}

\maketitle

\begin{abstract}
  In this paper we show the invariance of the Fredholm index of non-smooth pseudodifferential operators with coefficients in Hölder spaces. By means of this invariance we improve previous spectral invariance results for non-smooth pseudodifferential operators $P$ with coefficients in Hölder spaces. For this purpose we approximate $P$ with smooth pseudodifferential operators and use a spectral invariance result of smooth pseudodifferential operators.  Then we get the spectral invariance result in analogy to a proof of the spectral invariance result for non-smooth differential operators by Rabier.

\end{abstract}

 {\bf AMS-Classification:} 35 S 05, 47 A53, 47 G 30


\section{Introduction}

Non-smooth pseudodifferential operators arise naturally in the field of nonlinear partial differential equations. If the inverse of a non-smooth pseudodifferential operator $P$ exists we immediately get some existence results  for the partial differential equation $Pu=f$ where $f$ is a given suitable distribution. If we even can show, that the inverse of $P$ is a non-smooth pseudodifferential operator again, we immediately get regularity results for $Pu=f$ by means of the mapping properties of non-smooth pseudodifferential operators, see e.g.\  \cite{Marschall}. By means of spectral invariance results it is possible to reduce the proof of the invertibility of $P$ to a few cases. \\

In the literature there are already several spectral invariance results of pseudodifferential operators with smooth symbols $a$ in the Hörmander class $\Sn{m}{\rho}{\delta}$, cf.\ e.g.\  \cite{AlvarezHounie}, \cite{Beals}, \cite{Grubb},   \cite{Kryakvin}, \cite{SchroheLeopold1992}, \cite{SchroheLeopold1993}, \cite{LeopoldTriebel}, \cite{Schrohe1990}, \cite{Schrohe1992} and \cite{Ueberberg}. 
The Hörmander class $\Sn{m}{\rho}{\delta}$ consists of all smooth functions $a: \RnRn \rightarrow \C$ such that for all $k \in \N_0$
\begin{align*}
  |a|^{(m)}_k := \max_{|\alpha|,|\beta|\leq k} \sup_{x, \xi \in \R^n}|\pax{\alpha}{\beta} a(x,\xi)|\<{\xi}^{-(m-\rho|\alpha|+\delta|\beta|)} < \infty.
\end{align*}
All spectral invariance results listed before show, that the spectrum of pseudodifferential operators $a(x,D_x)$ with a symbol  $a \in \Sn{m}{\rho}{\delta}$
is independent of the choice of certain spaces.
 For instance the spectrum of $a(x, D_x):H^{m+s}_2(\Rn) \rightarrow H^s_2(\Rn)$ is independent of the choice of $s \in  \R$. Here $H^s_2(\Rn)$ denotes a Bessel potential space, defined in (\ref{BesselPotentialSpace}), Section \ref{section:Preliminaries} below. The associated pseudodifferential operator to a symbol $a$ in the Hörmander class is defined by 
\begin{align}\label{Def1}
		\op (a)u(x):= a(x,D_x) u(x) := \intr e^{ix \cdot \xi} a(x,\xi) \hat{u}(\xi) \dq \xi \qquad \forall u \in \s , x \in \Rn,
\end{align}
where $\s$ is
the Schwartz space, i.\,e.\,, the space of all rapidly decreasing smooth functions and $\hat{u}$ denotes the Fourier transformation of $u$. 
	There are only a few spectral invariance results in the non-smooth case.  For non-smooth differential operators 
$P(x,D_x):= \sum_{|\alpha| \leq m} A_{\alpha}(x) \p_x^{\alpha}$
 such a result was shown by  Rabier in \cite{Rabier}
for $A_{\alpha} \in \left( \mathcal{D}_0(\R^n) \right)^{l \times l}$, $|\alpha| \leq m$, $l \in \N$. 
Here $ \mathcal{D}_0(\R^n)$ consists of all continuous functions $a \in C^0_b(\R^n)$ with vanishing oscillation at infinity
\begin{align*}
	\mathcal{D}_0(\R^n):=\{ a \in C^0_b(\Rn): \lim_{x \rightarrow \infty} mc_{\alpha}(x)=0\}, 
\end{align*}
where
\begin{align*}
	mc_{\alpha}(x):= \sup\{|a(x)-a(y)|: |x-y| \leq 1\} \qquad \text{ for all } x\in \Rn.
\end{align*}
However, in applications not only non-smooth differential operators appear. A first spectral invariance result  for non-smooth pseudodifferential operators was shown in \cite{Paper2}.

In this paper we improve the spectral invariance result for pseudodifferential operators with non-smooth symbols in the symbol-class $C^{\tilde{m},\tau} S^m_{\rho,\delta}(\Rn \times \Rn, \mathscr{L}(\C^l))$, $l \in \N$, $m \in \R$, $0 \leq  \rho, \delta \leq 1$ which was shown in \cite{Paper2}. Here $C^{\tilde{m},\tau}$ denotes the Hölder space of the differentiation order $\tilde{m} \in \N_0$ with Hölder regularity $0< \tau<1$. Moreover we identify the space $\mathscr{L}(\C^l)$ with $\C^{l \times l}$. The symbol-class $C^{\tilde{m}, \tau} S^m_{\rho,\delta}(\Rn \times \Rn, M;  \mathscr{L}(\C^l))$, $M \in \N_0 \cup \{\infty\}$ and $m \in \R$, consists of all functions $a:\RnRn \rightarrow \C^{l \times l}$ fulfilling the following properties:  For all $\alpha, \beta \in \N_0^n$ with $|\beta| \leq \tilde{m}$, $|\alpha| \leq M$ we have
  \begin{itemize}

		\item[i)] $\p_x^{\beta} a(x, .) \in C^{M}(\Rn; \mathscr{L}(\C^l))$ for all $x \in \Rn$,

		\item[ii)] $\p_x^{\beta} \pa{\alpha} a \in C^{0}(\R^n_x \times \R^n_{\xi};\mathscr{L}(\C^l))$,

		\item[iii)] $\|\pa{\alpha} a(x, \xi)\|_{\mathscr{L}(\C^l)} \leq C_{\alpha}\<{\xi}^{m-\rho|\alpha|}$ for all $x,\xi \in \Rn$,

		\item[iv)] $\| \pa{\alpha} a(.,\xi)  \|_{C^{\tilde{m}, \tau}(\R^n;\mathscr{L}(\C^l))} \leq C_{\alpha}\<{\xi}^{m -\rho|\alpha|+\delta(\tilde{m}+\tau)}$ for all $\xi \in \R^n$.
	\end{itemize}
 For a given symbol $a$ we define the associated pseudodifferential operator as in the smooth case, cf.\  (\ref{Def1}). Moreover, if $M=\infty$, we also denote  $C^{\tilde{m}, \tau} S^m_{\rho,\delta}(\Rn \times \Rn, M;  \mathscr{L}(\C^l))$ by  $C^{\tilde{m}, \tau} S^m_{\rho,\delta}(\Rn \times \Rn;  \mathscr{L}(\C^l))$.

In contrast to the proof of the spectral invariance result of \cite{Paper2} we do not use the characterization of non-smooth pseudodifferential operators via iterated commutators in order to verify the spectral invariance statement. Instead we use the main idea of \cite{Rabier}, where Rabier showed a spectral invariance result of non-smooth differential operators. 
Hence 
we use the fact, that the set of all invertible operators is open within the set of all linear and bounded operators. Consequently
we approximate the non-smooth pseudodifferential operators with smooth pseudodifferential operators first and use the next spectral invariance result of smooth pseudodifferential operators afterwards:

\begin{thm}\label{thm:spectralInvarianceGlatterFall}
  Let $l \in \N$ and $a \in S^m_{\rho, \delta}(\RnRn; \mathscr{L}(\C^l))$ with $m \in \R$, $0 \leq \delta \leq \rho \leq  1$, $\rho>0$, $\delta <1$. If $a(x, D_x): \left( H^{s+m}_p \right)^l \rightarrow \left( H^{s}_p \right)^l$ is invertible for some $s \in \R$ and  $1  \leq  p< \infty$, where $p=2$ in case $\rho \neq 1$, then  
  $a(x, D_x): \left( H^{r+m}_q \right)^l \rightarrow \left( H^{r}_q \right)^l$ is invertible for all $r \in \R$ and  $1 \leq q< \infty$, where $q=2$ in case $\rho \neq 1$.
\end{thm}

\begin{proof}
  In the case $l =1$, $p,q \neq 1$ the result is a special case of \cite[Corollary 1.9]{Schrohe1990}.  Verifying the proof of  \cite[Corollary 1.9]{Schrohe1990}, we get the claim for general $p, q \in [1, \infty)$ in case $l=1$.
As Schrohe pointed out the result is also true for $l  \neq 1$. Using the case $l=1$ the general case easily follows from Cramers rule and the fact, that $ S^m_{\rho, \delta}(\RnRn)$ is closed with respect to pointwise multiplication. 
\end{proof}

On account of Theorem \ref{thm:spectralInvarianceGlatterFall} the invertibility of smooth pseudodifferential operators is independent of the choice of the Bessel potential space. 

Another important ingredient to show the spectral invariance result in the non-smooth case is the invariance of the Fredholm index of certain non-smooth pseudodifferential operators 
$P:H^{s+m}_p(\Rn) \rightarrow H^s_p(\Rn)$ with symbols in the class $ C^{\tilde{m},\tau} S^{m}_{\rho,\delta}(\RnRn;M)$. Since Fredholm operators are continuous by definition we restrict ourselves to all values of $s$ fulfilling the conditions of the continuity results of Marschall for non-smooth pseudodifferential operators, see \cite[Theorem 2.7 and Theorem 4.2]{Marschall}.  The next result states that the Fredholm index does not depend on the choice of $p \in [1,\infty)$ and of  all such $s$. The same holds true in the matrix-valued case.

\begin{thm}\label{thm:invarianceOfIndex}
  Let $l\in \N$,  $m \in \R$, $0\leq \delta <\rho \leq 1$, $ \tilde{m} \in \N$, $0<\tau <1$ with $\tilde{m}+\tau>\frac{1-\rho}{1-\delta}\cdot \frac{n}{2}$ if $\rho \neq 1$ and $M \in \N_0 \cup \{ \infty \}$ with $M \geq 2(n+1)$. Moreover let $a=(a_{i,j})_{i,j=1}^l \in C^{\tilde{m},\tau}\tilde{S}^{m}_{\rho,\delta}(\RnRn;M; \mathscr{L}(\C^l))$ be a symbol fulfilling the following properties for some $R>0$ and $C_0>0$:
  \begin{itemize}
	\item[0)] $a(\infty, \xi):= \lim_{|x| \rightarrow \infty }a(x,\xi)$ exists for all $\xi \in \Rn$.
    \item[1)] $|\det(a(x,\xi))|\<{\xi}^{-ml}\geq C_0$ for all $x,\xi \in \Rn$ with $|x|+|\xi|\geq R$.
	\item[2)] $\|\pa{\alpha} a(x, \xi) -  \pa{\alpha}a(\infty, \xi)\|_{\mathscr{L}(\C^l)}\<{\xi}^{-m+\rho|\alpha|} \xrightarrow{|x| \rightarrow \infty} 0$ uniformly in $\xi \in \Rn$ for all $\alpha \in \Non$ with $|\alpha| \leq n+2$.
  \end{itemize}
  Then for all $p \in [1,\infty)$ with $p=2$ if $\rho \neq 1$ and $s \in \R$ with $(1-\rho)\frac{n}{p}-(1-\delta)(\tilde{m}+\tau)<s<\tilde{m}+\tau$ 
  \begin{align*}
    A_p^s:=a(x, D_x): \left( H^{m+s}_p(\Rn) \right)^l \rightarrow \left( H^s_p(\Rn) \right)^l \qquad \text{is a Fredholm operator}
  \end{align*}
  and $\ind(A^s_p) = \ind(A^r_q)$ for all $q \in [1, \infty)$  with $q=2$ if $\rho \neq1$ and $(1-\rho)\frac{n}{p}-(1-\delta)(\tilde{m}+\tau)<r<\tilde{m}+\tau$.
\end{thm}
Here $ C^{\tilde{m},\tau}\tilde{S}^{m}_{\rho,\delta}(\RnRn;M; \mathscr{L}(\C^l)) \subseteq  C^{\tilde{m},\tau} S^{m}_{\rho,\delta}(\RnRn;M; \mathscr{L}(\C^l))$ consists of all so called \textit{slowly varying symbols}. For the definition of this symbol class we refer to Definition \ref{Def:SlowlyVaryingSymbols}. 
We prove this result in Section \ref{InvarianceFredholmIndex}. 

We would like to underline that in the previous theorem $\tilde{m}$ is not allowed to be $0$. 
The reason is simply that  we use a statement about  sufficient conditions for non-smooth pseudodifferential operators, see Theorem \ref{thm:Fredholmproperty},  in order to prove Theorem \ref{thm:invarianceOfIndex}. Theorem \ref{thm:Fredholmproperty} is a generalization of \cite[Theorem 1.1]{Paper3}.
Moreover we want to mention, that the condition 
\begin{align}\label{bedingung}
	(1-\rho)\frac{n}{p}-(1-\delta)(\tilde{m}+\tau)<s,r<\tilde{m}+\tau
\end{align}
of the previous lemma is independent of $p$ due to the following reason: 
\begin{itemize}
	\item  if $\rho=1$, then \eqref{bedingung} reduces to $-(1-\delta)(\tilde{m}+\tau)<s,r<\tilde{m}+\tau$
	\item if $\rho \neq 1$, then $p=2$ is assumed, therefore the condition \eqref{bedingung} becomes $(1-\rho)\frac{n}{2}-(1-\delta)(\tilde{m}+\tau)<s,r<\tilde{m}+\tau$ and is therefore again independent of $p$.
\end{itemize}
On account of shortness, in the future we always do not split up those two cases and we just write \eqref{bedingung} for it.

In the literature only few results concerning the invariance of the Fredholm index of pseudodifferential operators on $\Rn$ can be found. According to Schrohe, cf. \cite[Theorem 1.1]{Schrohe1992},  the Fredholm index of pseudodifferential operators $P:H^{st}_{\gamma} \rightarrow H^{st}_{\gamma}$ with so called slowly varying smooth symbols of certain Hörmander classes and order 0 is independent of $s,t$ and the admissible weight $\gamma$. For the definition of the weighted Sobolev space $H^{st}_{\gamma}$ we refer to \cite{Schrohe1992}. In the non-smooth case  there are just statements about the invariance of the Fredholm index for differential operators, cf. \cite{Rabier}, \cite{Rabier2008} so far. For general non-smooth pseudodifferential operators with coefficients in the Hölder space Theorem \ref{thm:invarianceOfIndex} is a first result about the invariance of the Fredholm index. 

Theorem \ref{thm:invarianceOfIndex} with $\tilde{m}=1$, $\rho=1$ and $\delta =0$ includes the  result about the invariance of the Fredholm index for non-smooth differential operators $P(x,D_x):= \sum_{|\alpha| \leq m} A_{\alpha}(x)\p_x^{\alpha}$, $A_{\alpha} \in \left( \mathcal{D}_0(\Rn)\right)^{l \times l}$, $l \in \N$, of Rabier in \cite[Theorem 3.5]{Rabier} if $A_{\alpha}  \in \left( C^{1,\tau}(\Rn) \right)^{l\times l}$ and if $\lim_{|x| \rightarrow \infty} a(x,\xi)$ exists. Since
\begin{itemize}
	\item Property $0)$ is fulfilled on account of the  assumption, that  $\lim_{|x| \rightarrow \infty} a(x,\xi)$ exists,
	\item Property $1)$ is a direct consequence of $P(x,D_x):  (H^{m}_p(\Rn))^l \rightarrow (L^p(\Rn)^l)$ being a Fredholm operator due to \cite[Theorem 2.1]{Rabier},
	\item Property  $2)$ holds due to the definition of  $\mathcal{D}_0(\Rn)$,
\end{itemize}
all assumptions of Theorem \ref{thm:invarianceOfIndex} hold for $P(x,D_x)$. 
Hence the Fredholmness of $P(x,D_x): (H^{m}_p(\Rn))^l \rightarrow (L^p(\Rn)^l)$ for some $p \in (1,\infty)$ implies the Fredholmness of $P(x,D_x): (H^{m}_q(\Rn))^l \rightarrow (L^q(\Rn)^l)$ for all $q \in (1,\infty)$  and the invariance of the Fredholm index with respect to $q \in (1,\infty)$  due to  Theorem \ref{thm:invarianceOfIndex}. This is exactly the claim of \cite[Theorem 3.5]{Rabier} in this case.

\makebox{} \newline

In the present paper we proceed as follows: For convenience of the reader, we give a short summary of all notations and function spaces needed later on in Section \ref{section:Preliminaries}. In Section \ref{section:PDO} a new subclass  $C_{unif}^{\tilde{m}, \tau} S^{m }_{\rho, \delta}(\Rn \times \Rn; \mathscr{L}(\C^l))$ of non-smooth symbols, belonging to the symbol-class $C^{\tilde{m}, \tau} S^m_{\rho,\delta}(\Rn \times \Rn; \mathscr{L}(\C^l))$, is introduced. Moreover we focus on proving an important ingredient for the main result of this paper: For every symbol $a \in C_{unif}^{\tilde{m}, \tau} S^{m }_{\rho, \delta}(\Rn \times \Rn; \mathscr{L}(\C^l))$ we show the existence of a sequence of smooth symbols which converge to $a$. In case $\delta =0$, the same statement even holds for every symbol of the symbol-class $C^{\tilde{m}, \tau} S^{m }_{\rho, 0}(\Rn \times \Rn; \mathscr{L}(\C^l))$. These results enable us to improve the spectral invariance result for non-smooth pseudodifferential operators in certain cases:

\begin{thm}\label{thm:SpectralInvarianceNonSmoothCase}
   Let $m \in \R$, $0 \leq \delta < \rho \leq 1$,  $0<\tau<1$ and  $\tilde{m} \in \N$  with $\tilde{m}+\tau>\frac{1-\rho}{1-\delta}\cdot \frac{n}{2}$ if $\rho \neq 1$. Moreover let  $a \in C^{\tilde{m}, \tau} \tilde{S}^{m}_{\rho, \delta}(\RnRn; \mathscr{L}(\C^l))$ for $ \delta=0$ and $a \in C_{unif}^{\tilde{m}, \tau} S^{m}_{\rho, \delta}(\RnRn; \mathscr{L}(\C^l) )\cap  C^{\tilde{m}, \tau} \tilde{S}^{m}_{\rho, \delta}(\RnRn; \mathscr{L}(\C^l))$ else be a symbol fulfilling properties 
0),
1) and 2) of Theorem \ref{thm:invarianceOfIndex}  for some $R>0$ and $C_0>0$.
Additionally we assume that  $A^s_p:=a(x,D_x):\left( H^{m+s}_p(\Rn) \right)^l \rightarrow \left( H^{m}_p(\Rn) \right)^l $ is invertible for some $p \in [1,\infty)$ and some $s \in \R$ with $(1-\rho)\frac{n}{p}-(1-\delta) (\tilde{m}+\tau) < s < \tilde{m}+\tau$. In case $\rho \neq 1$ we assume $p=2$.
   Then $A^r_q:=a(x,D_x):\left( H^{m+r}_q(\Rn) \right)^l \rightarrow \left( H^{m}_q(\Rn) \right)^l $ is invertible for every $q \in [1,\infty)$ and $r \in \R$ with $(1-\rho)\frac{n}{p}-(1-\delta)(\tilde{m}+\tau) < r< \tilde{m}+\tau$, where $q =2$ if $\rho \neq 1$.
\end{thm}

Hence, according to the previous theorem,  the invertibility of certain non-smooth pseudodifferential operators $P: H^{m+s}_p(\Rn)  \rightarrow H^{m}_p(\Rn)$ is not dependent of $p \in [1,\infty)$ and for all $s$, which fulfill the conditions of the continuity results for non-smooth pseudodifferential operators, c.f. \cite[Theorem 2.7 and Theorem 4.2]{Marschall}.  The same holds true in the operator-valued case. 
This statement is proved in Section \ref{section:spectralInvariance} 
while Section \ref{InvarianceFredholmIndex} deals to show the invariance of the Fredholm index for non-smooth pseudodifferential operators in certain cases, cf. Theorem \ref{thm:invarianceOfIndex}.
We want to emphasize, that we have to assume $\tilde{m}\neq 0$ in Theorem \ref{thm:SpectralInvarianceNonSmoothCase}, since we use Theorem \ref{thm:invarianceOfIndex} in the proof. 

In comparison to the spectral invariance results of \cite{Paper2} we do not need to restrict ourselves to the cases $\rho \in \{0,1\}$ and $\delta=0$ here. Additionally in case $(\rho, \delta)=(1,0)$ we have to assume much less smoothness with respect to the first variable of the symbol than in \cite{Paper2}: In \cite[Theorem 5.10]{Paper2} $\tilde{m}$ at least hast to be bigger than $\frac{3}{2} n +2$, while in Theorem \ref{thm:SpectralInvarianceNonSmoothCase} $\tilde{m} \in \N$ is arbitrary. 

However the spectral invariance results of \cite[Theorem 5.1 and Theorem 5.10]{Paper2} are not completely covered in the cases $\rho \in \{0,1\}$ and $\delta=0$ of Theorem \ref{thm:SpectralInvarianceNonSmoothCase}.  In Theorem \ref{thm:SpectralInvarianceNonSmoothCase} just symbols which are smooth with respect to the spatial variable $\xi$ are considered, while in  \cite[Theorem 5.1 and Theorem 5.10]{Paper2} the symbols are allowed to be non-smooth with respect to the spatial variable. Moreover in  Theorem \ref{thm:SpectralInvarianceNonSmoothCase} the claim holds not for all symbols of the classes $C^{\tilde{m}, \tau} S^m_{\rho, 0} (\RnRn ; \mathscr{L}(\mathbb{C}^l))$ for $\rho \in \{0,1\}$ and suitable variables $\tilde{m}, \tau$ and $m$ as in \cite[Theorem 5.1 and Theorem 5.10]{Paper2}, but just for symbols of a special subclass of that symbol class, namely for the so called slowly varying symbols, cf. Definition \ref{Def:SlowlyVaryingSymbols}. 

We now assume that a symbol $a$ fulfills all assumptions of Theorem 5.1 respectively Theorem 5.10 of \cite{Paper2} and of Theorem \ref{thm:SpectralInvarianceNonSmoothCase} with $\delta=0$ and $\rho \in \{0,1\}$. 
Then it becomes apparent that  Theorem \ref{thm:SpectralInvarianceNonSmoothCase} really is an improvement of  Theorem 5.1 respectively Theorem 5.10 of \cite{Paper2}. On account of  Theorem \ref{thm:SpectralInvarianceNonSmoothCase} we get the invertibility of $A^r_q$ for every $(1-\rho)\frac{n}{p}-(\tilde{m}+\tau) < r< \tilde{m}+\tau$, while the interval for suitable values for $r$ is smaller if Theorem 5.1 respectively Theorem 5.10 of \cite{Paper2} is applied. 
The difference of the interval length for suitable  values for the parameter $r$ is not less than $3n+2$ if $\rho=1$ and not less than $n$ if $\rho=0$. In addition while using  Theorem 5.1 respectively Theorem 5.10 of \cite{Paper2} we obtain the invertibility of $A^r_q$ is not for each choice of the symbol $a$ for all $q \in [1,\infty)$ as with Theorem \ref{thm:SpectralInvarianceNonSmoothCase}.\\

Theorem \ref{thm:SpectralInvarianceNonSmoothCase} includes the spectral invariance  result for non-smooth differential operators $P(x,D_x):= \sum_{|\alpha| \leq m} A_{\alpha}(x)\p_x^{\alpha}$, $A_{\alpha} \in \left( \mathcal{D}_0(\Rn)\right)^{l \times l}$, $l \in \N$, of Rabier in \cite[Theorem 4.4]{Rabier} in certain cases, but not completely. 
The invertibility of $P(x,D_x): (H^{m}_p(\Rn))^l \rightarrow (L^p(\Rn)^l)$ for some $p \in (1,\infty)$ implies the invertibility of $P(x,D_x): (H^{m}_q(\Rn))^l \rightarrow (L^q(\Rn)^l)$ for all $q \in (1,\infty)$  due to \cite[Theorem 4.4]{Rabier}. We additionally assume a bit more smoothness of the coefficients of the differential operator $P(x,D_x)$, namely $A_{\alpha} \in \left( C^{1,\tau}(\Rn) \right)^{l \times l}$, where $0<\tau <1$ and have to assume that  $\lim_{|x| \rightarrow \infty} a(x,\xi)$ exists. Then Theorem \ref{thm:SpectralInvarianceNonSmoothCase} with $\tilde{m}=1$, $\rho=1$ and $\delta =0$ already covers the result of \cite[Theorem 4.4]{Rabier}, since
\begin{itemize}
	\item $a \in C^{\tilde{m}, \tau} \tilde{S}^{m}_{\rho, \delta}(\RnRn; \mathscr{L}(\C^l))$ for $ \delta=0$ and $a \in C_{unif}^{\tilde{m}, \tau} S^{m}_{\rho, \delta}(\RnRn; \mathscr{L}(\C^l) )\cap  C^{\tilde{m}, \tau} \tilde{S}^{m}_{\rho, \delta}(\RnRn; \mathscr{L}(\C^l))$ else follows from the choice of $a$ and the definition of  $\mathcal{D}_0(\Rn)$,
	\item Property $0)$ is correct because of the assumption, that  $\lim_{|x| \rightarrow \infty} a(x,\xi)$ exists,
	\item Property $1)$ is a direct consequence of $P(x,D_x):  (H^{m}_p(\Rn))^l \rightarrow (L^p(\Rn)^l)$ being a Fredholm operator due to \cite[Theorem 2.1]{Rabier},
	\item Property $2)$ holds due to the definition of  $\mathcal{D}_0(\Rn)$.
\end{itemize}


\section{Notations and Function Spaces}\label{section:Preliminaries}

Throughout this paper we denote the set of all natural numbers without $0$ by $\N$. Additionally, we consider $n,l \in \N$ in this paper unless otherwise noted. The notations
\begin{itemize}
	\item $\lfloor x \rfloor:=\max\{m \in \Z: m \leq x\} \quad \text{ for each } x \in \R,$
	\item $\lceil x \rceil:=\min\{m \in \Z: m \geq x\} \quad \text{ for each } x \in \R,$
	\item $\<{x}:=(1+|x|^2)^{1/2} \quad \text{ for each } x \in \Rn$, 
	\item $\<{x;y}:=(1+|x|^2+|y|^2)^{1/2} \quad \text{ for each } x,y \in \Rn$, 
	\item $ \dq \xi:= (2 \pi)^{-n} d \xi$
\end{itemize}
are often used in this paper. 
For each multi-index $\alpha=(\alpha_1, \ldots, \alpha_n) \in \Non$ we define $ \p^{\alpha}_x := \p^{\alpha_1}_{x_1} \ldots \p^{\alpha_n}_{x_n} $. The linear hull of some functions $\Phi_1, \ldots,\Phi_k$ is denoted by $\Span\{\Phi_1, \ldots,\Phi_k\}$. 

For two Banach spaces $X,Y$ we denote the set of all linear and bounded operators $A:X \rightarrow Y$ by $\mathscr{L}(X,Y)$. We also write $\mathscr{L}(X)$ instead of $\mathscr{L}(X,X)$. Additionally we write $X'$ for the dual space of a Banach space $X$. Moreover the kernel respectively the image of an operator $A:X \rightarrow Y$ is denoted by $\ker(A)$ respectively $\Im(A)$.  

The set of all $k$-times differentiable and bounded functions is designated as $C^k_b(\Rn)$.

The \textit{Hölder space} $C^{0,\tau}(\Rn;\mathscr{L}(\C^l))$ of the differentiation order $0$ with Hölder continuity exponent $\tau \in (0,1]$ is the set of all matrix-valued functions $f: \Rn \rightarrow \C^{l \times l}$, $l \in \N$ fulfilling
\begin{align*}
  \| f\|_{C^{0,\tau}(\Rn;\mathscr{L}(\C^l))} :=\sup_{x \in \Rn} \|f(x)\|_{\mathscr{L}(\C^l)}+ \sup_{x \neq y} \frac{\|f(x)-f(y)\|_{\mathscr{L}(\C^l)}}{|x-y|^{\tau}}< \infty.
\end{align*}
A function $f: \Rn \rightarrow \C^{l \times l}$  is an element of the Hölder space $C^{\tilde{m},\tau}(\Rn;\mathscr{L}(\C^l))$ of the differentiation order $\tilde{m} \in \N_0$ if it is $\tilde{m}-$times differentiable and if $\p^{\alpha}_x f \in C^{0,\tau}(\Rn;\mathscr{L}(\C^l))$  for all $\alpha \in \Non$ with $|\alpha| \leq \tilde{m}$. Note that all Hölder spaces are Banach spaces. 
Using the definition of Hölder spaces one easily gets

\begin{lemma}\label{lemma:AbschatzungHolderraume}
  Let $0<t<\tau<1$ and $\tilde{m} \in \N_0$. Then 
  \begin{align*}
    \|f(x-y)-f(x)\|_{C^{\tilde{m},t} (\Rn_x)} \leq C |y|^{\tau-t} \|f\|_{C^{\tilde{m},\tau}} \qquad \text{ for all } f \in C^{\tilde{m},\tau}(\Rn), y \in \Rn. 
  \end{align*}
\end{lemma}
\begin{proof}
  For all $x,z \in \Rn$ with $x\neq z$ and  $\alpha \in \Non$ with $|\alpha| \leq \tilde{m}$ we obtain 
  \begin{align}\label{eq33a}
    \frac{| \p^{\alpha}_xf(x-y)-\p^{\alpha}_x f(x)- \p^{\alpha}_z f(z-y)+ \p^{\alpha}_z f(z)|}{ |x-z|^{t} |y|^{\tau -t}} 
    \leq 2 \| f \|_{C^{\tilde{m},\tau} } 
  \end{align}
    for all $f \in C^{\tilde{m},\tau}(\Rn), y \in \Rn \backslash \{0\}$ by using 
	$$|\p^{\alpha}_xf(x-y)-\p^{\alpha}_x f(x)- \p^{\alpha}_z f(z-y)+ \p^{\alpha}_z f(z)| \leq |\p^{\alpha}_xf(x-y)-\p^{\alpha}_z f(z-y)| + |\p^{\alpha}_xf(x)-\p^{\alpha}_z f(z)|$$
   if $|x-z|\leq |y|$ and 
      $$|\p^{\alpha}_xf(x-y)-\p^{\alpha}_x f(x)- \p^{\alpha}_z f(z-y)+ \p^{\alpha}_z f(z)| \leq |\p^{\alpha}_xf(x-y)-\p^{\alpha}_x f(x)| + |\p^{\alpha}_zf(z-y)-\p^{\alpha}_z f(z)|$$
	else. Hence
\begin{align}\label{eq333a}
	\sup_{x \neq z} \frac{| \p^{\alpha}_xf(x-y)-\p^{\alpha}_x f(x)- \p^{\alpha}_z f(z-y)+ \p^{\alpha}_z f(z)|}{ |x-z|^{t}} 
    \leq 2 \| f \|_{C^{\tilde{m},\tau} }  |y|^{\tau -t}
\end{align}
 for all $f \in C^{\tilde{m},\tau}(\Rn), y \in \Rn \backslash \{0\}$. 
 Using    
$$\|f(.-y)-f\|_{C_b^{\tilde{m}}(\Rn) } \leq \| f \|_{C^{\tilde{m},\tau-t} } |y|^{\tau-t} \leq C \| f \|_{C^{\tilde{m},\tau} } |y|^{\tau-t} \qquad \text{for all }f \in C^{\tilde{m},\tau}(\Rn) $$ 
  and inequality (\ref{eq333a}) provides the claim.
\end{proof}

By means of the definition of the H\"older spaces and the Leibniz-rule we additionally obtain:
\begin{lemma}\label{lemma:PropertyHoelderSpaces}
  Let $\tilde{m} \in \N_0$, $0< \tau < 1$, $\Omega \subseteq \Rn$ be closed and $f,g \in C^{\tilde{m},\tau}(\Omega)$. Then 
  \begin{align*}
    \|fg\|_{C^{\tilde{m},\tau}(\Omega)} \leq \sum_{\tilde{m}_1+\tilde{m}_2=\tilde{m}} C_{\tilde{m}}\left\{ \|f\|_{C^{\tilde{m}_1}_b(\Omega)} \|g\|_{C^{\tilde{m}_2, \tau}(\Omega)} + \|f\|_{C^{\tilde{m}_1, \tau}(\Omega)} \|g\|_{C^{\tilde{m}_2}_b(\Omega)} \right\}.
  \end{align*}
\end{lemma}

	With interpolation results at hand, the possibility often arises to verify some convergences of symbol-families, cf. e.g. the statement about the approximation of certain non-smooth symbols with smooth ones, see Lemma \ref{lemma:KonvergenzFürDeltaUngleich0} below. Another example is the auxiliary lemma, which is needed in order to verify the regularity result for non-smooth pseudodifferential operators, see Lemma \ref{HilfslemmaFuerLemma3.3} below. 
We need the next five interpolation properties for H\"older spaces throughout this paper:

\begin{lemma}\label{lemma:InterpolationResult}
  Let $k, \tilde{m} \in \N$ with $k \leq \tilde{m}$, $0<\tau < 1$ and $\theta:= \frac{k}{\tilde{m} + \tau}$. Then 
  \begin{align*}
    \|f\|_{C^k_b(\Rn)} \leq C\|f\|^{1-\theta}_{C^0_b(\Rn)} \|f\|^{\theta}_{C^{\tilde{m}, \tau}(\Rn)} \qquad \text{for all } f \in C^{\tilde{m}, \tau}(\Rn).
  \end{align*}
\end{lemma}
For the proof of the previous lemma we refer to \cite[Lemma 2.41]{Diss}.

\begin{lemma}\label{lemma:interpolation1} 
   Let $k,\tilde{m} \in \N$ and  $0<s, \tau<1$ with $k+s < \tilde{m}+\tau$. Then there are constants $C_{\tilde{m}, k}, C_{\theta}, C_k>0$ such that for $\theta= \frac{k+s}{\tilde{m}+\tau}$ we have
  \begin{itemize}
    \item[i)] $\|f\|_{C^k_b(\Rn \backslash B_1(0))} \leq  C_{\tilde{m}, k} \|f\|_{C^0_b(\Rn \backslash B_1(0))}^{1-\frac{k}{\tilde{m}}}  \|f\|_{C_b^{\tilde{m}}(\Rn \backslash B_1(0))}^{\frac{k}{\tilde{m}}}\quad \forall f \in C_b^{\tilde{m}}(\Rn \backslash B_1(0))$, 
    \item[ii)] $\|f\|_{C^{k,s}(\Rn \backslash B_1(0))} \leq  C_{\theta} \|f\|_{C_b^0(\Rn \backslash B_1(0))}^{1-\theta}  \|f\|_{C^{\tilde{m}, \tau}(\Rn \backslash B_1(0))}^{\theta} \quad \forall f \in C^{\tilde{m},\tau}(\Rn \backslash B_1(0))$,
    \item[iii)] $\max\limits_{|\beta|=k} \|\p_x^{\beta} f\|_{C_b^0(\Rn)} \leq C_k \|f\|^{1-\frac{k}{k+1}}_{C_b^0(\Rn)} \left( \max\limits_{|\alpha|=k+1} \|\p_x^{\alpha}f\|_{C_b^0(\Rn)} \right)^{\frac{k}{k+1}} \quad \forall f \in C_b^{k+1}(\Rn)$.
  \end{itemize}
Here $B_1(0)$ denotes the open ball in $\Rn$ around $0$ with radius 1. 
\end{lemma}
\begin{proof}[Sketch of proof:]
  We start with verifying $i)$. Let $ \Rn_+:=\{(x', x_n) \in \R^{n-1}\times \R: x_n>0\}$. We define the extension operator $E: C_b^0(\overline{\Rn_+})\rightarrow C_b^0(\Rn)$ for all $f \in C_b^0(\overline{\Rn_+})$ and $x=(x', x_n) \in \Rn$ by 
  \begin{align*}
    Ef(x):= \left\{ 
     \begin{array}{ll}
	f(x) & \text{if } x_n\geq 0,\\
	a_1 f(x', -x_n) + a_2 f(x', -2x_n) + \ldots + a_{\tilde{m}+1} f(x', -(\tilde{m}+1)x_n) & \text{else}.
     \end{array}
\right.
  \end{align*}
  Analogous to \cite[Chapter 1]{Lunardi} one can show the existence of $a_1, \ldots, a_{\tilde{m}+1} \in \R$ such that $E \in \mathscr{L}(C_b^k(\overline{\Rn_+}), C_b^k(\Rn)) $ for all $k \in\{0, \ldots, \tilde{m}\}$. This implies for $\Omega= \Rn_+$:
  \begin{align}\label{ee1}
    E \in \mathscr{L}(C_b^{\theta}(\overline{\Omega}), C_b^{\theta}(\Rn)) \qquad \text{for all } \theta \in [0, \tilde{m}].
  \end{align}
  By means of the standard localization argument it can be shown that (\ref{ee1}) holds for all open sets $\Omega \subseteq \Rn$ with bounded $C_b^{\tilde{m}}-$ boundary. Now we choose $\Omega:= \Rn \backslash \overline{B_1(0)}$. The restriction operator 
  \begin{align*}
    R(f):= f|_{\overline{\Omega}} \qquad \text{for all } f \in C^0(\Rn)
  \end{align*}
  obviously belongs  to $\mathscr{L}(C_b^k(\Rn); C_b^k(\overline{\Omega}))$ for all $k \in \{0, \ldots, \tilde{m}\}$ and because of an interpolation result, see e.g.\  \cite[Remark 1.3.4]{Lunardi} there is a constant $C_{\tilde{m}, k}>0$ such that
  \begin{align*}
    \|f\|_{C_b^k(\Rn \backslash B_1(0))} 
    &= \|R(Ef)\|_{C_b^k(\Rn \backslash B_1(0))} 
    \leq C_{\tilde{m}, k} \|Ef\|_{C_b^0(\Rn)}^{1-\frac{k}{\tilde{m}}} \|Ef\|_{C_b^{\tilde{m}}(\Rn)}^{\frac{k}{\tilde{m}}}\\
    &\leq C_{\tilde{m}, k} \|f\|_{C_b^0(\Rn  \backslash B_1(0))}^{1-\frac{k}{\tilde{m}}} \|f\|_{C_b^{\tilde{m}}(\Rn  \backslash B_1(0))}^{\frac{k}{\tilde{m}}} \qquad \text{for all } f \in C_b^{\tilde{m}}(\Rn \backslash B_1(0)).
  \end{align*}
  Claim $ii)$ can be proved in a similar way. It remains to show $iii)$. This is done by mathematical induction with respect to $k$. For $k=1$ one gets by means of the Taylor expansion formula for all $i\in \{1, \ldots, n\}$ and $h>0$: 
  \begin{align*}
    |\p_{x_i}f(x)| 
    &\leq \frac{|f(x+he_j)-f(x)|}{h} + \|\p_{x_i}\p_{x_j}f\|_{C_b^0(\Rn)} \cdot h \\
    &\leq \frac{2\|f\|_{C_b^0(\Rn)}}{h} +  \max_{|\alpha|=2} \|\p_x^{\alpha}f\|_{C_b^0(\Rn)} \cdot h, \qquad f \in C_b^{k+1}(\Rn). 
  \end{align*}
  Choosing $h= \|f\|_{C_b^0(\Rn)}^{1/2} \left(  \max_{|\alpha|=2} \|\p_x^{\alpha}f\|_{C_b^0(\Rn)} \right)^{-1/2}$ and taking the maximum over $x \in \Rn$ and $i\in \{1, \ldots, n\}$ proves $iii)$ for $k=1$.
  Assuming, that $iii)$ already holds for $k \in \N$ we obtain due to the case $k=1$:
  \begin{align*}
    &\max_{|\beta|=k+1} \|\p_x^{\beta}f\|_{C_b^0(\Rn)} 
    \leq  C \max_{|\alpha|=k} \|\p_x^{\alpha}f\|_{C_b^0(\Rn)}^{1/2} \left(  \max_{|\alpha|=k+2} \|\p_x^{\alpha}f\|_{C_b^0(\Rn)} \right)^{1/2} \\
    &\qquad \qquad \leq C \left\{ \|f\|^{1-\frac{k}{k+1}}_{C_b^0(\Rn)} \left( \max\limits_{|\beta|=k+1} \|\p_x^{\beta}f\|_{C_b^0(\Rn)} \right)^{\frac{k}{k+1}} \right\}^{1/2}
	  \left(  \max_{|\alpha|=k+2} \|\p_x^{\alpha}f\|_{C_b^0(\Rn)} \right)^{1/2}
  \end{align*}
  for all $f \in C_b^{k+2}(\Rn)$. Dividing the previous inequality through the second term of the right side provides:
  \begin{align*}
    &\max_{|\beta|=k+1} \|\p_x^{\beta}f\|_{C_b^0(\Rn)} 
    \leq C_{k+1} \|f\|^{1-\frac{k+1}{k+2}}_{C_b^0(\Rn)}  \left(  \max_{|\alpha|=k+2} \|\p_x^{\alpha}f\|_{C_b^0(\Rn)} \right)^{\frac{k+1}{k+2}} 
  \end{align*}
  for all $f \in C_b^{k+2}(\Rn)$, which proves the last statement. 
\end{proof}

\begin{lemma}\label{lemma:interpolation2} 
  Let $\tilde{m} \in \N$ and $\alpha \in \Non$ with $|\alpha| \leq \tilde{m}$ be arbitrary. Then there is a constant $C_{|\alpha|}>0$ such that  for $\theta =\frac{|\alpha|}{\tilde{m}}$ we have
  \begin{align*}
    \sup_{x \in \Rn \backslash B_R(0)} \left| \p_x^{\alpha} f(x) \right| \leq C_{|\alpha|} \|f\|^{1-\theta}_{C_b^0(\Rn \backslash B_R(0))}  \|f\|^{\theta}_{C_b^{\tilde{m}}(\Rn \backslash B_R(0))} \qquad \forall f\in C_b^{\tilde{m}}(\Rn), R\geq 1.
  \end{align*}

\end{lemma}
\begin{proof}
   Let $\alpha \in \Non$ with $|\alpha| \leq \tilde{m}$ be arbitrary and $\theta :=\frac{|\alpha|}{\tilde{m}}$. We define for all $f \in C^{\tilde{m}}(\Rn)$ and each $R\geq 1$ the function $f_R:\Rn \backslash B_1(0) \rightarrow \C$ by
   \begin{align*}
      f_R(x):= f(Rx) \qquad \text{for all } x\in \Rn \backslash B_1(0).
   \end{align*}
  On account of $\p_x^{\alpha} f_R(x) = R^{|\alpha|}(\p_x^{\alpha}f)_R(x)$ for all $x \in \Rn \backslash B_1(0)$ we obtain $(\p_x^{\alpha} f_R)(R^{-1}x) \cdot R^{-|\alpha|}=\p_x^{\alpha}f(x)$  for all $x \in \Rn \backslash B_R(0)$. Hence we get together with Lemma \ref{lemma:interpolation1}  
  \begin{align*}
    \sup_{x \in \Rn \backslash B_R(0)} \left| \p_x^{\alpha} f(x)\right| 
    &\leq R^{-|\alpha|} \|\p_x^{\alpha} f_R\|_{L^{\infty}( \Rn \backslash B_1(0))}\\
    &\leq C_{|\alpha|, \tilde{m}}  R^{-|\alpha|} \| f_R\|_{C_b^0( \Rn \backslash B_1(0))}^{1-\theta} \| f_R\|_{C_b^{\tilde{m}}( \Rn \backslash B_1(0))}^{\theta}\\
     &\leq C_{|\alpha|, \tilde{m}}  \| f\|_{C_b^0( \Rn \backslash B_R(0))}^{1-\theta} \| f\|_{C_b^{\tilde{m}}( \Rn \backslash B_R(0))}^{\theta} 
  \end{align*}
  for all $R\geq 1, f\in C_b^{\tilde{m}}(\Rn)$.
\end{proof}

\begin{lemma}\label{lemma:interpolation3} 
  Let $\tilde{m} \in \N$ and $\alpha \in \Non$ with $|\alpha| \leq \tilde{m}$. Additionally let $0<s<\tau <1$. Then there are constants $C_{|\alpha|, \tilde{m}}, C_{\theta}>0$ such that for $\theta:= \frac{|\alpha|+s}{\tilde{m}+\tau}$ we have
  \begin{align*}
    \|\p_x^{\alpha} f\|_{C^{0,s}(\Rn \backslash B_R(0))} \leq C_{|\alpha|, \tilde{m}} &\|f\|_{C^0_b(\Rn \backslash B_R(0))}^{1-\frac{|\alpha|}{\tilde{m}}} \|f\|_{C^{\tilde{m}}_b(\Rn \backslash B_R(0))}^{\frac{|\alpha|}{\tilde{m}}}\\
    &+ C_{\theta}  \|f\|_{C^0_b(\Rn \backslash B_R(0))}^{1-\theta} \|f\|_{C^{\tilde{m}, \tau}(\Rn \backslash B_R(0))}^{\theta} 
  \end{align*}
  for all $f\in C^{\tilde{m}, \tau}(\Rn)$ and $ R\geq 1$.
\end{lemma}
\begin{proof}
  Let $\alpha \in \Non$ with $|\alpha| \leq \tilde{m}$ and $0<s<\tau$ be arbitrary. We define  $\theta:= \frac{|\alpha|+s}{\tilde{m}+\tau}$. Moreover we define for each $f \in C^{\tilde{m}, \tau}(\Rn)$ and each $R\geq 1$ the function $f_R:\Rn \backslash B_1(0) \rightarrow \C$ as in the proof of Lemma \ref{lemma:interpolation2}. Because of $(\p_x^{\alpha} f_R)(R^{-1}x) \cdot R^{-|\alpha|}=\p_x^{\alpha}f(x)$  for all $x \in \Rn \backslash B_R(0)$ we get
  \begin{align*}
    \|\p_x^{\alpha} f\|_{C^{0,s}(\Rn \backslash B_R(0))} 
    &=R^{-|\alpha|}\sup_{x\in \Rn\backslash B_R(0)} \left| \p_x^{\alpha} f_R(R^{-1}x) \right| \\
    &\qquad + R^{-|\alpha|} \sup_{ \substack{x,y \in \Rn\backslash B_R(0) \\ x\neq y} } \frac{|(\p_x^{\alpha}f_R)(R^{-1}x)-(\p_y^{\alpha}f_R)(R^{-1}y)|}{|x-y|^s}\\
    &\leq R^{-|\alpha|} \|f_R\|_{C_b^{|\alpha|}(\Rn\backslash B_1(0))} + R^{-|\alpha|-s} \|f_R\|_{C^{|\alpha|, s}(\Rn\backslash B_1(0) )}.
  \end{align*}
  An application of Lemma \ref{lemma:interpolation1} yields
  \begin{align*}
      &\|\p_x^{\alpha} f\|_{C^{0,s}(\Rn \backslash B_R(0))}
      \leq R^{-|\alpha|} C_{|\alpha|, \tilde{m}} \|f_R\|^{1-\frac{|\alpha|}{\tilde{m}}}_{C_b^{0}(\Rn\backslash B_1(0))} \|f_R\|^{\frac{|\alpha|}{\tilde{m}}}_{C_b^{\tilde{m}}(\Rn\backslash B_1(0))}\\
      &\qquad \qquad \qquad \qquad \qquad + R^{-|\alpha|-s} \|f_R\|^{1-\theta}_{C_b^{0}(\Rn\backslash B_1(0) )} \|f_R\|^{\theta}_{C^{\tilde{m}, \tau}(\Rn\backslash B_1(0) )}\\
      &\qquad \quad \leq C_{|\alpha|, \tilde{m}} \|f\|_{C^0_b(\Rn \backslash B_R(0))}^{1-\frac{|\alpha|}{\tilde{m}}} \|f\|_{C^{\tilde{m}}_b(\Rn \backslash B_R(0))}^{\frac{|\alpha|}{\tilde{m}}}
    + C_{\theta}  \|f\|_{C^0_b(\Rn \backslash B_R(0))}^{1-\theta} \|f\|_{C^{\tilde{m}, \tau}(\Rn \backslash B_R(0))}^{\theta}.
  \end{align*}
\end{proof}

Using the definition of the Hölder spaces 
we obtain: 
\begin{bem}\label{bem:MatrixwertigerHölderRaum}
  Let $0 <\tau<1$, $\tilde{m} \in \N_0$ and $l \in \N$. Then we have some constants $C_1, C_2>0$ such that
  \begin{align*}
    \|f\|_{C^{\tilde{m},\tau}(\Rn; \C^{l\times l})} \leq C_1 \max_{i,j = 1, \ldots, l} \| f_{ij}\|_{C^{\tilde{m},\tau}(\Rn)}\leq  C_2\|f\|_{C^{\tilde{m},\tau}(\Rn; \C^{l\times l})}
	\end{align*}
 for all $ f=(f_{ij})_{i,j=1}^l \in C^{\tilde{m},\tau}(\Rn; \C^{l\times l})$.
\end{bem}

 The \textit{Bessel potential space} $H^s_p(\Rn)$, $s \in \R$ and $1<p<\infty$ is defined by 
\begin{align}\label{BesselPotentialSpace} 
  H^s_p(\Rn):= \{ f \in \sd: \<{D_x}^s f \in L^p(\Rn)  \}
\end{align}
where $\<{D_x}^s:=\op(\<{\xi}^s) $.\\

A generalization of the Bessel potential spaces are the \textit{Triebel-Lizorkin spaces} $F^s_{p,q}(\Rn)$ with $s \in \R$ and $0 < p,q < \infty$. For the definition of these spaces we define a dyadic partition of unity $(\varphi_j)_{j \in \N_0}$ in the usual way: Let $\varphi_0 \in C^{\infty}_c(\Rn)$ with $\varphi_0(\xi)=1$ if $|\xi|\leq 1$ and $\varphi_0(\xi)=0$ if $|\xi|\geq 2$. Then 
   \begin{align*}
      \varphi_j(\xi):= \varphi_0(2^{-j}\xi)-\varphi_0(2^{-j-1}\xi) \qquad \text{for all } \xi \in \Rn, j \in \N.
   \end{align*}
The Triebel-Lizorkin space  $F^s_{p,q}(\Rn)$ with $s \in \R$ and $0 < p,q < \infty$ is defined by
\begin{align*}
	F^s_{p,q}(\Rn)&:=\{ f \in \sd: \|f\|_{F^s_{p,q}} < \infty \}, \qquad \text{where}\\
	 \|f\|_{F^s_{p,q}}&:= \left\| \left(\sum_{j=0}^{\infty} 2^{qjs} |\varphi_j(D_x) f(x) |^q\right)^{1/q} \right\|_{L^p(\Rn_x)} .
\end{align*}
Note, that the quasi-Banach space $F^s_{p,q}(\Rn)$ is independent of the choice of the dyadic partition of unity  $(\varphi_j)_{j \in \N_0}$, see e.g.\  \cite{Tr} or \cite{RunstSickel} for more information about these spaces. We just want to mention those properties which are needed in this paper:
\begin{itemize}
	\item $F^s_{p,q}(\Rn)$ is even a Banach space, if $p,q \geq 1$,
	\item $F^s_{p,2}(\Rn)= H^s_p(\Rn)$ for all $1<p<\infty$,
\end{itemize} 
To get a better readability of the paper, we use the notation of \cite[p.922]{Marschall} and define $H^s_1(\Rn):=F^s_{1,2}(\Rn)$ for all $s \in \R$. 
Note that there are different definitions in the literature.
The definition of $F^s_{p,q}(\Rn)$ can not be extended for $p=\infty$, since then the space would not be independent of the choice of the dyadic partition of unity  $(\varphi_j)_{j \in \N_0}$. However, with 
\begin{align*}
	Q_{j,l}:= \{x=(x_1, \ldots, x_n) \in \Rn: 2^{-j}l_i \leq x_i \leq 2^{-j}(l_i+1), i=1, \ldots, n \} 
\end{align*}
 for all $ j \in \N_0$ and $ l \in \Z^n$,  the space $F^s_{\infty,q}(\Rn)$,  $s \in \R$ and $0 < q < \infty$,  is defined in the following way, see e.g.\  \cite[Chapter 2]{RunstSickel}:
\begin{align*}
	F^s_{\infty,q}(\Rn)&:=\{ f \in \sd: \|f\|_{F^s_{\infty,q}} < \infty \}, \qquad \text{where}\\
	 \|f\|_{F^s_{\infty,q}}&:= \sup_{j \in \N_0} \sup_{l \in \Z^n} \left(  2^{jn} \int_{Q_{j,l}} \left( \sum_{k=1}^{\infty}  2^{qks} |\varphi_k(D_x) f(x) |^q\right)dx  \right)^{1/q} .
\end{align*}
According to \cite[Section 2.1.5]{RunstSickel}, we have 
\begin{align}\label{eq:DualspaceOfTribelLizorkinSpace}
	\left( H^s_1(\Rn) \right)'=F^{-s}_{\infty, 2}(\Rn) \qquad \text{for all } s\in \R.
\end{align} 

Moreover we mention some notations concerning the symbol-classes: In case $l=1$ we write $C^{\tilde{m},s} S^m_{\rho,\delta}(\Rn \times \Rn, M)$, $M \in \N_0 \cup \{\infty \}$, and $S^m_{\rho,\delta}(\Rn \times \Rn)$ instead of $C^{\tilde{m},s} S^m_{\rho,\delta}(\Rn \times \Rn, M; \mathscr{L}(\C^l))$ and $S^m_{\rho,\delta}(\Rn \times \Rn;\mathscr{L}(\C^l))$ respectively. Additionally $C^{\tilde{m},s} S^m_{\rho,\delta}(\Rn \times \Rn, M ; \mathscr{L}(\C^l))$ can be considered as a Fr\'{e}chet space with respect to the semi-norms
\begin{align*}
  |a|^{(m)}_{k, C^{\tilde{m},s} S^m_{\rho, \delta}}:= \max_{|\alpha| \leq k} \sup_{\xi \in \Rn} \left\{ \|\pa{\alpha} a(.,\xi) \|_{C^{\tilde{m},s}(\Rn)} \<{\xi}^{-m+ \rho|\alpha| - \delta (\tilde{m}+s)} + Se_{\alpha}(\xi) \right\},
\end{align*}
for all $k \in \N_0$ 
with $k \leq M$ 
and $a \in C^{\tilde{m},s} S^m_{\rho, \delta}(\Rn \times \Rn, M)$, 
where $Se_{\alpha}(\xi)=0$ if $\delta =0$ and 
  $Se_{\alpha}(\xi) = \|\pa{\alpha} a(.,\xi) \|_{C^0_b(\Rn)} \<{\xi}^{-m+ \rho|\alpha|}$
else, in the case $l=1$ and 
\begin{align*}
  |a|^{(m)}_{k, C^{\tilde{m},s} S^m_{\rho, \delta}(\Rn \times \Rn, M; \mathscr{L}(\C^l))}:= \max_{i,j=1, \ldots, l}  |a_{i,j}|^{(m)}_{k, C^{\tilde{m},s} S^m_{\rho, \delta}},
\end{align*}
for all $k \in \N_0$ and $a =(a_{i,j})_{i,j=1}^l \in C^{\tilde{m},s} S^m_{\rho, \delta}(\Rn \times \Rn; \mathscr{L}(\C^l))$. 
In particular the symbol class $C^{\tilde{m},s} S^m_{\rho,\delta}(\Rn \times \Rn, M; \mathscr{L}(\C^l))$ is a Banach space if $N \neq \infty$ with norm $  |a|^{(m)}_{M, C^{\tilde{m},s} S^m_{\rho, \delta}(\Rn \times \Rn, M; \mathscr{L}(\C^l))}$.
\\

One of the two main results of this paper is the invariance of the Fredholm index for certain non-smooth pseudodifferential-operators, cf.\ Theorem \ref{thm:invarianceOfIndex}. Hence we finally add the definition of an Fredholm operator and of the Fredholm index:

\begin{defn}
  Let $X,Y$ be Banach spaces and $T \in \mathscr{L}(X,Y)$. Then $T$ is called Fredholm operator if its kernel  $\ker (T)$ is finite dimensional and if its range $\Im(T)$ is closed and has finite co-dimension, i.e.~, there is a finite dimensional subspace $Z\subseteq Y$ such that $Y= \operatorname{Im}(T)\oplus Z$.
\end{defn}

\begin{defn}
	 Let $X,Y$ be Banach spaces and $T:X \rightarrow Y$ be a Fredholm operator. Then the Fredholm index of $T$ is defined by
\begin{align*}
	\operatorname{ind}(T):= \dim(\operatorname{ker} (T) ) - \operatorname{codim}(\operatorname{Im}(T) ).
\end{align*}
\end{defn}

\subsection{Space of Amplitudes and Oscillatory Integrals}\label{subsection:SpaceOfAmplitudes}

For the definition of pseudodifferential operators we need the so-called oscillatory integrals. They are defined for all elements of the \textbf{space of amplitudes} $\mathscr{A}^{m,N}_{\tau,M}(\RnRn)$, $ N,M \in \N_0 \cup \{ \infty \}$, $m,\tau \in \R$. A function  $a:\Rn \times \Rn \rightarrow \C$ is in the set $\mathscr{A}^{m,N}_{\tau,M}(\RnRn)$, $ N,M \in \N_0 \cup \{ \infty \}$, $m,\tau \in \R$, if for all
$\alpha, \beta \in \Non$ with $|\alpha| \leq N$, $|\beta| \leq M$ we have
  \begin{enumerate}
    \item[i)] $\p^{\alpha}_{\eta} \p^{\beta}_{y} a(y,\eta) \in C^0(\RnRnx{y}{\eta})$,
    \item[ii)] $\left|\p^{\alpha}_{\eta} \p^{\beta}_{y} a(y, \eta) \right| \leq C_{\alpha, \beta} (1 + |\eta|)^m (1 + |y|)^{\tau}$ for all $y, \eta \in \Rn$,
  \end{enumerate} 
where the existence of all occuring derivatives is implicitly assumed. The  \textbf{oscillatory integral} of $a \in \mathscr{A}^{m,N}_{\tau,M}(\RnRn)$ is defined by 
\begin{align} \label{DefOsziInt}
  \osint e^{-iy \cdot \eta} a(y,\eta) dy \dq \eta := \lim_{\e \rightarrow 0} \iint \chi(\e y, \e \eta) e^{-iy \cdot \eta} a(y,\eta) dy \dq \eta, 
\end{align}
where $\chi \in \mathcal{S}(\RnRn)$ with $\chi(0,0)=1$. 
\makebox{}\\

For all $m \in \N$ we define 
\begin{align*}
  A^m(D_{x},\xi) &:= \<{\xi}^{-m} \<{D_x}^{m} \qquad \qquad  \qquad \qquad  \qquad \qquad \qquad \quad   \text{ if } m \text{ is even},\\
  A^m(D_{x},\xi) &:= \<{\xi}^{-m-1} \<{D_x}^{m-1} -\sum_{j=1}^n \<{\xi}^{-m} \frac{\xi_j}{\<{\xi} } \<{D_x}^{m-1} D_{x_j} \quad \text{ else}.
\end{align*}
We now summarize all properties of the oscillatory integral needed in this paper. For the proof of those results we refer to  \cite[Section 2.1]{Paper3}.

\begin{thm}\label{thm:propertiesOsciInt}
  Let $m, \tau \in \R$ and $N,M \in \N_0 \cup \{ \infty \}$ with $N>n + \tau$. Moreover let $l, l' \in \N$ with $N \geq l'> n+\tau$ and $M \geq l > n+m$. 
  Then the oscillatory integral (\ref{DefOsziInt}) exists for all $a \in \mathscr{A}^{m,N}_{\tau,M}(\RnRn)$ and we have for all $l_1,l_2 \in \N$ with $l_1\leq N$ and $l_2 \leq l$:
  \begin{align*}
    \osint e^{-iy\cdot \eta} a(y,\eta) dy \dq \eta &= \iint e^{-iy\cdot \eta} A^{l'}(D_{\eta}, y) A^l(D_y, \eta) a(y,\eta) dy \dq \eta, \\
     \osint e^{-iy\cdot \eta} a(y,\eta) dy \dq \eta &=  \osint e^{-iy\cdot \eta}A^{l_1}(D_{\eta}, y) A^{l_2} (D_y, \eta) a(y,\eta) dy \dq \eta.
  \end{align*}
\end{thm}

\begin{thm}
  Let $m, \tau \in \R$, $m_i, \tau_i \in \R$ for $i\in \{1,2\}$ and $N \in \N_0 \cup \{ \infty \}$ such that there is a $l' \in \N$ with  $N\geq l'>n + \tau$. Moreover let $\alpha, \beta \in\Non$ with $|\alpha| \leq \tilde{M}$, where $\tilde{M}:= \max \{ \hat{m} \in \N_0: N-\hat{m} >n+\tau \}$ and $l \in \N$ with $l>m+n$. Considering $a \in C^0(\Rn_y \times \Rn_{y'} \times \Rn_{\eta} \times \Rn_{\xi})$ with 
  \begin{itemize}
    \item $\left| A^{l'}(D_{\eta}, y) A^l(D_y, \eta) a(y, y', \eta, \xi) \right| \leq C_{l, l'} \<{y}^{\tau-l'} \<{\eta}^{m-l} \<{y'}^{\tau_1} \<{\xi}^{m_1}$,
    \item $\left| A^{l'}(D_{\eta}, y) A^l(D_y, \eta) \pa{\alpha} \p_{y'}^{\beta} a(y, y', \eta, \xi) \right| \leq C_{l, l', \alpha, \beta} \<{y}^{\tau-l'} \<{\eta}^{m-l} \<{y'}^{\tau_2} \<{\xi}^{m_2}$ 
  \end{itemize}
  for all $y,y', \eta, \xi \in \Rn$ we have for all  $y', \xi \in \Rn$:
  \begin{align*}
     \pa{\alpha} \p_{y'}^{\beta} \osint e^{-iy \cdot \eta} a(y, y', \eta, \xi) dy \dq \eta 
      = \osint e^{-iy \cdot \eta}  \pa{\alpha} \p_{y'}^{\beta} a(y, y', \eta, \xi) dy \dq \eta.
  \end{align*}
\end{thm}

\begin{kor}
  Let $m, \tau \in \R$ and $N \in \N_0 \cup \{ \infty \}$ such that there is an $l' \in \N$ with  $N\geq l'>n + \tau$. Moreover let $l \in \N$ with $l>n+m$. Additionally let $a_j, a \in C^0(\RnRn)$, $j \in \N_0$ such that for all $\alpha, \beta \in\Non$ with $|\alpha| \leq N$ and $|\beta| \leq l$ the derivatives $\p_{\eta}^{\alpha} \p_y^{\beta} a_j, \p_{\eta}^{\alpha} \p_y^{\beta} a$ exist in the classical sense and 
  \begin{itemize}
    \item $|\p_{\eta}^{\alpha} \p_y^{\beta} a_j(y, \eta)| \leq C_{\alpha, \beta} \<{\eta}^{m} \<{y}^{\tau}$ for all $\eta, y \in \Rn$, $j \in \N_0$,
    \item $|\p_{\eta}^{\alpha} \p_y^{\beta} a(y, \eta)| \leq C_{\alpha, \beta} \<{\eta}^{m} \<{y}^{\tau}$ for all $\eta, y \in \Rn$,
    \item $\p_{\eta}^{\alpha} \p_y^{\beta} a_j(y, \eta) \xrightarrow{j \rightarrow \infty } \p_{\eta}^{\alpha} \p_y^{\beta} a(y, \eta)$ for all $\eta, y \in \Rn$.
  \end{itemize}
  Then 
  \begin{align*}
    \lim_{j \rightarrow \infty} \osint e^{-iy \cdot \eta}  a_j(y, \eta) dy \dq \eta 
    =  \osint e^{-iy \cdot \eta}  a(y, \eta) dy \dq \eta .
  \end{align*}
\end{kor}


\section{Pseudodifferential Operators and their Properties}\label{section:PDO}

The goal of this section is to discuss all properties of pseudodifferential operators needed later on. In particular we prove for every $a \in C^{\tilde{m},\tau} S^{m}_{\rho, \delta}(\RnRn)$ with $\delta=0$ the existence of a sequence of smooth symbols $(a_{\e})_{0< \e \leq 1}$, which converge to $a$. Additionally in case $\delta \neq 0$ we show the same result for non-smooth symbols of a certain subclass of $C^{\tilde{m},\tau} S^{m}_{\rho, \delta}(\RnRn)$. 
	This sequence enables us to use the spectral invariance result for smooth pseudodifferential operators in order to prove that one for non-smooth pseudodifferential operators. 
The sequence $(a_{\e})_{0< \e \leq 1}$, fulfilling those properties, is defined in the next remark:

\begin{bem}\label{bem:p_eIstGlattesSymbol}
    Let $0 <\tau <1$, $\tilde{m} \in \N_0$, $m \in \R$ and $0 \leq \rho,\delta \leq 1$. Additionally let $(\varphi_{\e})_{\e>0}$ be a positive Dirac-family and $a \in C^{\tilde{m},\tau} S^{m}_{\rho, \delta}(\RnRn)$. For all $\e \in (0,1]$ we define $a_{\e}: \RnRn \rightarrow \C$ by
    \begin{align*}
      a_{\e}(x,\xi):= \left(a(., \xi) \ast \varphi_{\e} \right)(x):=\intr a(y,\xi) \varphi_{\e}(x-y) dy \qquad \text{ for all } x, \xi \in \Rn.
    \end{align*}
    Then $a_{\e} \in S^{m}_{\rho, \delta}(\RnRn)$ for all $\e \in (0,1]$.
\end{bem}

The previous remark can be proved 
by means of the properties of the convolution and the Dirac-family. The second ingredient for reaching the aim of this section is the next boundedness result:

\begin{thm}\label{thm:BoundednessResultNonSmooth}
  Let $m \in \R$, $0 \leq \delta \leq \rho \leq 1$ with $\rho >0$, $1\leq p \leq \infty$, $l\in \N_0$, $N\in \N$ with $N > \max\left\{ n/2, n/p \right\}$, $\tilde{m} \in \N_0$ and $0<\tau <1$. Additionally let $\tilde{m}+\tau > \frac{1-\rho}{1-\delta}\cdot \frac{n}{2}$ if $\rho <1$. Denoting $k_p:=(1-\rho) n \left| 1/2 - 1/p \right|$ and let $(1-\rho)n/p - (1-\delta)(\tilde{m}+ \tau) < s< \tilde{m}+\tau$ we get for all $a \in C^{\tilde{m}, \tau} S^{m-k_p}_{\rho, \delta} (\RnRn,N;  \mathscr{L}(\C^l))$  the boundedness of 
$$a(x,D_x): \left( H^{s+m}_p(\Rn) \right)^l \rightarrow \left( H^s_p(\Rn)\right)^l.$$
Moreover we get for some $k \in \N_0$
 with $k<N$
 and some $C_{s,l}>0$, independent of $a \in C^{\tilde{m}, \tau} S^{m-k_p}_{\rho, \delta} (\RnRn, N; \mathscr{L}(\C^l))$, the following estimate:
  \begin{align}\label{eq:Continuity}
    \left\| a(x,D_x) f \right\|_{\left( H^s_p(\Rn)\right)^l} \leq C_{s,l} |a|^{(m-k_p)}_{k, C^{\tilde{m}, \tau} S^{m-k_p}_{\rho, \delta} (\RnRn,N; \mathscr{L}(\C^l))} \| f\|_{\left(H_p^{s+m}\right)^l}
  \end{align}
  for all $f \in  \left(H_p^{s+m}(\Rn)\right)^l, a \in C^{\tilde{m}, \tau} S^{m-k_p}_{\rho, \delta} (\RnRn,N; \mathscr{L}(\C^l))$.
\end{thm}

We already proved the previous boundedness result for $p \neq 1$ and $l=1$ in   \cite[Theorem 3.7]{Paper1}.  For $p=1$ and $l=1$ we can show the claim in nearly the same way by means of the Theorem of Banach Steinhaus, see e.g. \cite[Theorem 2.5]{Rudin}. We just have to extend the family of linear and continuous functions on which Banach Steinhaus should be applied. For convenience of the reader we nevertheless prove this case, now:


\begin{proof}
  Case $l=1$ is already proved in \cite[Theorem 3.7]{Paper1} for $p\neq1$. Now let $p=1$ and $l=1$. 
Due to \cite[Theorem 4.2]{Marschall} it just remains to proof \eqref{eq:Continuity}.

	 As already mentioned, this can be done similarly as for $p \neq 1$. Hence we first define for 
 $f\in H^{s+m}_1(\Rn)$ and $g \in \left(H^s_1(\Rn)\right)'$ with $\|f\|_{H^{s+m}_1} \leq 1$ and $\|g\|_{\left(H^s_1\right)'} \leq 1$ the operator $\op_{f,g} :  C^{\tilde{m}, \tau} S^{m-k_1}_{\rho, \delta} (\RnRn,N;  \mathscr{L}(\C^l)) \rightarrow \C$ by 
$$\op_{f,g}(a):= \skh{a(x, D_x)f}{g}{H^s_1, \left(H^s_1\right)'}.$$
 Using the boundedness of  
$a(x,D_x):  H^{s+m}_1(\Rn) \rightarrow  H^s_1(\Rn)$ for each symbol $a  \in  C^{\tilde{m}, \tau} S^{m-k_1}_{\rho, \delta} (\RnRn,N;  \mathscr{L}(\C^l))$  we get the existence of a constant $C$, independent of  $f \in H^{s+m}_1(\Rn), g \in  \left(H^s_1(\Rn)\right)' \text{ with } \|f\|_{H^{s+m}_1} \leq 1 \text{ and } \|g\|_{ \left(H^s_1\right)'} \leq 1$, such that 
  \begin{align*}
    |\skh{a(x, D_x)f}{g}{H^s_1, \left(H^s_1\right)'}| 
    &\leq \left\| a(x, D_x)f \right\|_{H^s_1} \left\| g \right\|_{\left(H^s_1\right)'}
    &\leq C \left\| a(x, D_x)\right\|_{\mathscr{L}(H^{s+m}_1; H^s_1) }.
  \end{align*}
  Consequently 
an application of  the theorem of Banach-Steinhaus, cf.\;e.g.\;\cite{Rudin}, Theorem 2.5 provides that
  \begin{align*}
    \left\{ \op_{f,g} : f \in H^{s+m}_1(\Rn), g \in  \left(H^s_1(\Rn)\right)' \text{ with } \|f\|_{H^{s+m}_1} \leq 1 \text{ and } \|g\|_{ \left(H^s_1\right)'} \leq 1 \right\}
  \end{align*}
  is equicontinuous. Therefore there is  a $k \in \N$ with $k \leq N$ and a $C>0$ such that
  \begin{align*}
     |\op_{f,g}(a)| \leq C | a |^{(m)}_{k} \quad \text{for all }  &a \in C^{\tilde{m}, \tau} S^{m-k_1}_{\rho, \delta} (\RnRn,N;  \mathscr{L}(\C^l)), f \in H^{s+m}_1(\Rn),\\
 & g \in  \left(H^s_1(\Rn)\right)' \text{ with } \|f\|_{H^{s+m}_1} \leq 1 \text{ and } \|g\|_{ \left(H^s_1\right)'} \leq 1.
  \end{align*}
 This implies the claim:
  \begin{align*}
    \| a(x, D_x) \|_{\mathscr{L}(H^{s+m}_1;H^s_1)} 
    &= \sup_{\|f\|_{H^{s+m}_1} \leq 1} \| a(x, D_x)f \|_{H^s_1}
    = \sup_{\|f\|_{H^{s+m}_1} \leq 1} \sup_{\|g\|_{ \left(H^s_1\right)'} \leq 1} |\op_{f,g}(a)| \\
    &\leq C |a|^{(m)}_{k} \qquad \text{for all } a \in  C^{\tilde{m}, \tau} S^{m-k_1}_{\rho, \delta} (\RnRn,N;  \mathscr{L}(\C^l)).
  \end{align*}

 Considering the components of $a(x,D_x)$ we obtain the general case $l \in \N$ and $1 \leq p < \infty$ by means of case $l=1$.
\end{proof}

With these results at hand, we now are able to show

\begin{lemma}\label{DichtheitGlatteSymbolklasseInNonSmoothOne}
    Let $a = (a^{ij})_{i,j=1}^l \in C^{\tilde{m},\tau} S^{m}_{\rho, 0}(\RnRn;\mathscr{L}(\C^l))$ with $0 <\tau<1$, $\tilde{m} \in \N_0$, $m \in \R$. We set for each $\e \in (0, 1]$ the function $a_{\e}:= \left( a^{ij}_{\e} \right)_{i,j=1}^l \in S^{m}_{\rho, 0}(\RnRn; \mathscr{L}(\C^l))$, where $a^{ij}_{\e}$ are defined as in Remark \ref{bem:p_eIstGlattesSymbol} for each $i,j \in \{ 1, \ldots, l \}$. Then
    for all $0<t<\tau$
    \begin{align*}
      a_{\e} \xrightarrow{\e \rightarrow 0} a \qquad \text{ in }  C^{\tilde{m},t} S^{m}_{\rho, 0}(\RnRn; \mathscr{L}(\C^l)).
    \end{align*}

\end{lemma}

\begin{proof}
  First let us assume $l =1$. Let $0<t<\tau$ and $\alpha \in \Non$ be arbitrary. By means of Lemma \ref{lemma:AbschatzungHolderraume} we obtain
  \begin{align*}
    \|\pa{\alpha} a(x-y, \xi) - \pa{\alpha} a(x, \xi) \|_{C^{\tilde{m}, t}(\Rn_x) } \<{\xi}^{-m +\rho |\alpha|} 
    \leq C_{\alpha} |a|^{(m)}_{|\alpha|, C^{\tilde{m},\tau} S^m_{\rho, 0}} |y|^{\tau-t}
    \rightarrow 0
  \end{align*}
  uniformly in $\xi$ for $|y|\rightarrow 0$.
  Let $k \in \N_0$ be arbitrary now. In order to show $|a_{\e}-a|^{(m)}_{k, C^{\tilde{m},t} S^m_{\rho, 0}} \rightarrow 0$ for $\e \rightarrow 0$ we choose an arbitrary $\tilde{\e}>0$. On account of the previous convergence there is a $\tilde{\delta}>0$ such that for all $|\alpha| \leq k$ we have
  \begin{align}\label{eq1}
    \|\pa{\alpha} a(x-y, \xi) - \pa{\alpha} a(x, \xi) \|_{C^{\tilde{m}, t}(\Rn_x) } \<{\xi}^{-m +\rho |\alpha|}  \leq \frac{\tilde{\e}}{2} \qquad \text{ for all } |y|< \tilde{\delta}, \xi \in \Rn. 
  \end{align}
  The properties of a positive Dirac-family provides the existence of a $\nu > 0 $ such that 
  \begin{align}\label{eq2}
    \left| \int_{|y| \geq \tilde{\delta} } \varphi_{\e} (y) dy \right| < \frac{\tilde{\e}}{4A} \qquad \text{ for all } \e < \nu, 
  \end{align}
  where $A:= \max \left\{ 1, \max_{|\alpha|\leq k} \sup_{\xi \in \Rn} \left( \|\pa{\alpha} a(., \xi)\|_{C^{\tilde{m}, t}} \<{\xi}^{-m+\rho |\alpha|} \right) \right\}$.
   Because of the properties of the convolution and of a positive Dirac-family we get
  \begin{align}\label{eq0}
    & |a_{\e}-a|^{(m)}_{k, C^{\tilde{m},t} S^m_{\rho, 0} }
    \leq \max_{|\alpha| \leq k} \sup_{\xi \in \Rn} \left\{ \|\pa{\alpha} \left( a_{\e}(x,\xi) - a(x,\xi) \right) \|_{C^{\tilde{m},t}(\Rn_x)} \<{\xi}^{-m + \rho |\alpha|} \right\} \nonumber \\ 
    &\leq \max_{|\alpha| \leq k} \sup_{\xi \in \Rn} \left\{ \max_{|\beta| \leq \tilde{m}} \int \left\| \pa{\alpha} D_x^{\beta} a(x-y,\xi) - \pa{\alpha} D_x^{\beta} a(x,\xi) \right\|_{C^{0,t}(\Rn_x)} \varphi_{\e}(y) dy  \<{\xi}^{-m + \rho |\alpha|} \right\} \nonumber\\ 
    & \leq \max_{|\alpha| \leq k} \sup_{\xi \in \Rn} \left\{ \int \left\| \pa{\alpha}  a(x-y,\xi) - \pa{\alpha}  a(x,\xi) \right\|_{C^{\tilde{m},t}(\Rn_x)} \varphi_{\e}(y) dy  \<{\xi}^{-m + \rho |\alpha|} \right\}.
  \end{align}
  Splitting the integral of the previous inequality into two with respect to the sets $\{|y| < \tilde{\delta} \}$ and $\{|y| \geq \tilde{\delta} \}$ respectively we obtain the claim in case $l=1$ by using the inequalities (\ref{eq1}) - (\ref{eq2}):
  \begin{align*}
    |a_{\e}-a|^{(m)}_{k, C^{\tilde{m},t} S^m_{\rho, 0} } 
    \leq \tilde{\e}.
  \end{align*}
  The general case can be verified by using Remark \ref{bem:MatrixwertigerHölderRaum} and case $l=1$ for each entry $a^{ij} \in C^{\tilde{m},\tau} S^{m}_{\rho, 0}(\RnRn)$ of $a = (a^{ij})_{i,j=1}^l \in C^{\tilde{m},\tau} S^{m}_{\rho, 0}(\RnRn; \mathscr{L}(\C^l))$.
\end{proof}

Unfortunately we cannot modify the previous proof for general symbols of the symbol-class $C^{\tilde{m},\tau} S^m_{\rho, \delta}(\RnRn; \mathscr{L}(\C^l))$, $\delta \neq 0$,
since assumption $ii)$ of the definition below is needed for the proof in general. 
For the set of all those symbols we introduce the new symbol-class  $C_{unif}^{\tilde{m}, \tau} S^{m }_{\rho, \delta}(\Rn \times \Rn; \mathscr{L}(\C^l))$:

\begin{Def}
  Let $0 \leq \rho,\delta \leq 1$, $0 <\tau<1$, $\tilde{m} \in \N_0$ and $m \in \R$. Then the symbol-class $C_{unif}^{\tilde{m}, \tau} S^{m }_{\rho, \delta}(\Rn \times \Rn; \mathscr{L}(\C^l))$ is the set of all functions $a: \RnRn \rightarrow \C^{l \times l}$ such that for all $\alpha \in \Non$ we have
  \begin{itemize}
    \item[i)] $a \in C^{\tilde{m},\tau} S^{m}_{\rho, \delta}(\RnRn; \mathscr{L}(\C^l))$,
    \item[ii)] $\lim\limits_{|h| \rightarrow 0} \sup\limits_{\xi \in \Rn} \left\| \pa{\alpha} \left( a(x+h,\xi)-a(x, \xi) \right) \right\|_{C^0_b(\Rn; \mathscr{L}(\C^l)) } \<{\xi}^{-m+\rho|\alpha|} = 0$.
  \end{itemize}
\end{Def}

\begin{lemma} \label{lemma:KonvergenzFürDeltaUngleich0}
    Let $0 <\tau<1$, $\tilde{m} \in \N_0$, $m \in \R$ and $0 \leq \rho,\delta \leq 1$ with $\delta \neq 0$. For all $a = (a^{ij})_{i,j=1}^l \in C_{unif}^{\tilde{m},\tau} S^{m}_{\rho, \delta}(\RnRn; \mathscr{L}(\C^l))$ and all $\e >0$ we set $a_{\e}:= \left( a^{ij}_{\e} \right)_{i,j=1}^l \in S^{m}_{\rho, \delta}(\RnRn; \mathscr{L}(\C^l))$, where $a^{ij}_{\e}$ are defined as in Remark \ref{bem:p_eIstGlattesSymbol} for each $i,j \in \{ 1, \ldots, l \}$. Then for all $0<t<\tau$ we have 
    \begin{align*}
      a_{\e} \xrightarrow{\e \rightarrow 0} a \qquad \text{ in }  C^{\tilde{m},t} S^{m}_{\rho, \delta}(\RnRn; \mathscr{L}(\C^l)).
    \end{align*}
\end{lemma}

\begin{proof}
   First we prove case $l=1$. Let $0<t<\tau$ and $\alpha \in \Non$ be arbitrary.
   Due to an interpolation result, cf.\  \cite[Corollary 1.2.18]{LunardiSemigoups} and \cite[Theorem 1.3.3]{TriebelInterpol}  we get for $\theta =\frac{m+t}{m+\tau}$:
  \begin{align*}
    &\| \pa{\alpha} a(x+h,\xi) - \pa{\alpha} a(x,\xi)\|_{C^{\tilde{m},t}(\Rn_x)} \<{\xi}^{-m+\rho |\alpha|-\delta (\tilde{m}+t)} \\
    &\leq C_{\theta}  \| \pa{\alpha} a(x+h,\xi) - \pa{\alpha} a(x,\xi)\|^{1-\theta}_{C^{0}_b(\Rn_x)}  \| \pa{\alpha} a(x+h,\xi) - \pa{\alpha} a(x,\xi)\|^{\theta}_{C^{\tilde{m},\tau}(\Rn_x)} \<{\xi}^{-m+\rho |\alpha|-\delta (\tilde{m}+t)} \\
    & \leq C_{\theta} \left( |a|^{(m)}_{|\alpha|, C^{\tilde{m},\tau} S^m_{\rho, \delta}} \right)^{\theta}
      \| \pa{\alpha} a(x+h,\xi) - \pa{\alpha} a(x,\xi)\|^{1-\theta}_{C^{0}_b(\Rn_x)} \<{\xi}^{(-m+\rho |\alpha|)(1-\theta)}
    \rightarrow 0
  \end{align*}
  uniformly in $\xi$ for $|h| \rightarrow 0$.\\
  Now let $k \in \N_0$ be arbitrary.
  In order to prove $|a_{\e}-a|^{(m)}_{k, C^{\tilde{m},t} S^m_{\rho, \delta}} \rightarrow 0$ for $\e \rightarrow 0$ we choose an arbitrary $\tilde{\e}>0$. Then the previous estimate implies the existence of a $\delta_1>0$ such that for all $|\alpha| \leq k$ we have 
  \begin{align}\label{eq3}
    \| \pa{\alpha} a(x+h,\xi) - \pa{\alpha} a(x,\xi)\|_{C^{\tilde{m},t}(\Rn_x)} \<{\xi}^{-m+\rho |\alpha|-\delta (\tilde{m}+t)}  \leq \frac{ \tilde{\e} }{4} \quad \forall |h| \leq \delta_1, \xi \in \Rn.
  \end{align}
  Since $a \in C_{unif}^{\tilde{m},\tau} S^{m}_{\rho, \delta}(\RnRn)$, there is a $\delta_2>0$ such that
  \begin{align}\label{eq4}
    \sup_{\xi \in \Rn} \left\| \pa{\alpha} ( a(x+h, \xi) -  a(x, \xi) )\right\|_{C^0_b(\Rn)} \<{\xi}^{-m + \rho |\alpha|} \leq \frac{\tilde{\e}}{4} \quad  \forall  |h| \leq \delta_2, |\alpha| \leq k.
  \end{align}
  The properties of a positive Dirac-family provide the existence of a $\nu >0$ such that
  \begin{align}\label{eq5}
    \left| \int_{|y| \geq \min{ \{\delta_1, \delta_2 \} } } \varphi_{\e} (y) dy \right| < \frac{\tilde{\e}}{4A} \qquad \text{ for all } \e < \nu, 
  \end{align}
  where 
  \begin{align*}
    A &:= \max \left\{ 1, 2 \max_{|\alpha| \leq k} \sup_{\xi \in \Rn} \left( \|\pa{\alpha} a(., \xi)\|_{C^{\tilde{m}, t}} \<{\xi}^{-m+\rho |\alpha|-\delta(\tilde{m}+t)} \right), \right.  \\
    & \qquad \qquad \qquad \qquad \qquad \qquad \qquad  \left. 2 \max_{|\alpha| \leq k}\sup_{\xi \in \Rn} \left( \|\pa{\alpha} a(., \xi)\|_{C^0_b} \<{\xi}^{-m+\rho |\alpha|} \right) \right\}.
  \end{align*}
  In the same way as equality (\ref{eq0}) in the proof of Lemma \ref{DichtheitGlatteSymbolklasseInNonSmoothOne} we can show
  \begin{align*}
     &|a_{\e}-a|^{(m)}_{k, C^{\tilde{m},t} S^m_{\rho, \delta} }\\
     &\quad \leq \max_{|\alpha| \leq k} \sup_{\xi \in \Rn} \left\{ \int \left\| \pa{\alpha}  a(x-y,\xi) - \pa{\alpha}  a(x,\xi) \right\|_{C^{0}_b(\Rn_x)} \varphi_{\e}(y) dy  \<{\xi}^{-m + \rho |\alpha|} \right. \\
    &\qquad \qquad \left. +  \int \left\| \pa{\alpha}  a(x-y,\xi) - \pa{\alpha}  a(x,\xi) \right\|_{C^{\tilde{m},t}(\Rn_x)} \varphi_{\e}(y) dy  \<{\xi}^{-m + \rho |\alpha|-\delta(\tilde{m}+t)} \right\}.
  \end{align*}
  If we split the first integral of the previous inequality into two over the sets $\{|y| < \delta_2 \}$ and $\{|y| \geq \delta_2 \}$, respectively, and if we additionally split the second integral of the previous inequality into two with respect to the sets $\{|y| < \delta_1 \}$ and $\{|y| \geq \delta_1 \}$, respectively, we obtain the claim in case $l=1$ by using the inequalities (\ref{eq3}) - (\ref{eq5}):
  \begin{align*}
    &|a_{\e}-a|^{(m)}_{k, C^{\tilde{m},t} S^m_{\rho, \delta} } \leq \tilde{\e}.
  \end{align*}
  The general case can be verified by using Remark \ref{bem:MatrixwertigerHölderRaum} and case $l=1$ for each entry $a^{ij} \in C_{unif}^{\tilde{m},\tau} S^{m}_{\rho, \delta}(\RnRn)$ of $a = (a^{ij})_{i,j=1}^l \in C^{\tilde{m},\tau} S^{m}_{\rho, \delta}(\RnRn; \mathscr{L}(\C^l))$. 
\end{proof}

Moreover we mention the continuity result for smooth pseudodifferential operators needed later on. For the proof we refer to \cite[Theorem 2.7]{KumanoGo}:

\begin{thm}\label{thm:ConituityResultSmoothCase}
  Let $0\leq \delta<\rho \leq 1$, $m \in \R$ and $1 \leq p < \infty$. Considering a symbol $a \in S^m_{\rho,\delta}(\RnRn)$, we obtain for all $s \in \R$ the continuity of 
  \begin{align*}
    a(x, D_x): H^{m+s}_p(\Rn) \rightarrow H^s_p(\Rn). 
  \end{align*}
Moreover, we have for some $k \in \N_0$
\begin{align*}
	\|a(x, D_x)u\|_{H^{s}_p} \leq C |a|^{(m)}_{k} \|u\|_{H^{m+s}_p} \qquad \text{ for all } u \in H^{m+s}_p(\Rn).
\end{align*}
\end{thm}

We also need the following subclass of non-smooth symbols:
\begin{Def}
 	Let  $m\in \R$, $0 \leq \rho, \delta \leq 1$ and $M \in \N_0 \cup \{\infty\}$. Then a function $a: \RnRn \rightarrow\C$ is in the symbol-class $S^m_{\rho, \delta}(\RnRn; M)$, if for all $\alpha, \beta \in \Non$ with $|\alpha| \leq M$ there is a constant $ C_{\alpha, \beta }>0$ such that
\begin{itemize}
	\item $\p^{\beta}_x a(x,.) \in C^M(\Rn)$,
	\item $\p^{\beta}_x \pa{\alpha} a \in C^0(\Rn_x\times \Rn_{\xi})$,
	\item $|\pa{\alpha}\p_x^{\beta} a(x,\xi)| \leq C_{\alpha, \beta } \<{\xi}^{m-\rho|\alpha|+\delta|\beta|} \qquad \text{for all } x,\xi \in \Rn.$
\end{itemize}
For $l \in \N$ the symbol $a=(a_{i,j})_{i,j=1}^l \in S^m_{\rho, \delta}(\RnRn; M;\mathscr{L}(\C^l))$, if $a_{i,j} \in S^m_{\rho, \delta}(\RnRn; M)$ for each $i,j \in \{1, \ldots, l\}$. 
\end{Def}

Similarly to \cite[Remark 4.2]{Diss} we obtain by interpolation the following embedding of two non-smooth symbol-classes:
 \begin{align}\label{eq:GlattesSymbolIstNichtglatt}
      S^{m}_{\rho, \delta}(\Rn \times \R^n; M) \subseteq C^{\tilde{m},\tau} S^{m}_{\rho, \delta}(\Rn \times \R^n;M).
  \end{align}
for all $0<\tau < 1$, $\tilde{m} \in \N_0$, $m\in \R$, $M \in \N_0 \cup \{ \infty \}$ and $0 \leq \rho,\delta \leq 1$.

Additionally we get by means of interpolation the next estimate for non-smooth symbols:

\begin{lemma}\label{lemma:AbschatzungNichtglattesSymbol}
  Let $\tilde{m} \in \N_0$, $0<\tau < 1$, $0 \leq \delta, \rho \leq 1$, $m \in \R$ and $a \in C^{\tilde{m}, \tau} S^m_{\rho, \delta}(\RnRn; M)$. 
  Then we get for all $\alpha \in \Non$ with $|\alpha| \leq M$ and $k \in \N_0$ with $k \leq \tilde{m}$:
  \begin{align*}
    \|\pa{\alpha} a(., \xi)\|_{C^k_b(\Rn)} \leq C_{\alpha, \beta} \<{\xi}^{\tilde{m}-\rho|\alpha| + \delta k} \qquad \text{for all } \xi \in \Rn.
  \end{align*}
\end{lemma}

\subsection{Symbol-Smoothing}\label{Subsection:SymbolSmoothing}

Results as the Fredholm property of non-smooth pseudodifferential operators can be proved by means of symbol-smoothing, see e.g.\  \cite{Paper3}. 
In the present paper we use this tool in order to show a regularity result for non-smooth pseudodifferential operators, cf. Lemma \ref{lemma:Lemma3.3}.
 Some properties for symbol-smoothing can be found in \cite[Section 3]{Paper3} and if the symbol is smooth with respect to the second variable we refer to  \cite[Section 1.3]{Taylor2}. 
In order to define the symbol-smoothing for non-smooth pseudodifferential operators, 
 we fix the dyadic partition of unity $(\psi_j)_{j \in N_0}$ defined as in Section \ref{section:Preliminaries} and  $\phi \in C^{\infty}_c(\Rn)$ with $\phi(\xi)=1, |\xi|\leq 1$, throughout the whole subsection.
 
Using 
\begin{align}\label{8e}
  C_1 \<{\xi}^{-a} \leq 2^{-ja} \leq C_2 \<{\xi}^{-a} \qquad \text{for all } \xi \in \supp(\psi_j), j \in \N
\end{align}
for $a \in \R$
 we get for all $\alpha \in \Non$, $j \in \N_0$:
\begin{align}\label{7e}
  \|\p^{\alpha}_{\xi}\psi_j\|_{\infty} \leq C_{\alpha}\<{\xi}^{-|\alpha|}.
\end{align}
 Additionally  the operator $J_{\e}$ is defined for all $\e>0$ via
\begin{align*}
  J_{\e}:= \phi(x, D_x).
\end{align*}

\begin{Def}\label{Def:SymbolSmoothing}
  Let $\tilde{m} \in \N_0$, $0<\tau <1$, $M \in \N_0 \cup \{ \infty\}$, $m \in \R$ and $0 \leq \delta < \rho \leq 1$. For $\gamma \in (\delta, \rho)$ we set $\e_j:= 2^{-j\gamma}$. For each $a \in C^{\tilde{m}, \tau} S^{m}_{\rho, \delta}(\RnRn; M)$ we define 
  \begin{itemize}
    \item $a^{\sharp}(x, \xi):= \sum\limits_{j=0}^{\infty} J_{\e_j} a(x, \xi)\psi_j(\xi)$ for all $x, \xi \in \Rn$,
    \item $a^b(x, \xi):= a(x, \xi)-a^{\sharp}(x, \xi)$ for all $x, \xi \in \Rn$.
  \end{itemize}
\end{Def}

For so called slowly varying symbols $a$,  the  symbols $a^{\sharp}(x,\xi)$ and $a^b(x,\xi)$ have very useful properties, needed later on.

\begin{Def}\label{Def:SlowlyVaryingSymbols}
  Let $\tilde{m} \in \N_0$, $0 < \tau < 1$, $m \in \R$, $0 \leq \delta, \rho \leq 1$ and $M \in \N_0 \cup \{ \infty \}$. 
  Then  $a \in C^{\tilde{m}, \tau} S^m_{\rho, \delta} (\RnRn; M)$ belongs to the symbol-class $C^{\tilde{m}, \tau} \dot{S}^m_{\rho, \delta} (\RnRn; M)$, if for all $\alpha, \beta \in \Non$ with $|\alpha| \leq M$ and $|\beta| \leq \tilde{m}$ we have
  \begin{align}\label{eq:SlowlyVaryingSymbols}
    |\pa{\alpha} D_x^{\beta} a(x, \xi)| \leq C_{\alpha, \beta}(x) \<{\xi}^{m-\rho|\alpha|+\delta|\beta|} \qquad \text{for all } x, \xi \in \Rn,
  \end{align}
  where $C_{\alpha, \beta}(x)$ is a bounded function, which converges to zero, if $|x| \rightarrow \infty$. \\
  Moreover, $a \in C^{\tilde{m}, \tau} S^m_{\rho, \delta} (\RnRn; M)$ belongs to the symbol-class $C^{\tilde{m}, \tau} \tilde{S}^m_{\rho, \delta} (\RnRn; M)$, if for all $\beta \in \Non$ with $|\beta| \leq \tilde{m}$ and $|\beta | \neq 0$ we have
  \begin{align*}
    D_x^{\beta} a(x, \xi) \in C^{\tilde{m}-|\beta|, \tau} \dot{S}^{m+\delta|\beta|}_{\rho, \delta} (\RnRn; M).
  \end{align*}
	Additionally all symbols $a \in S^m_{\rho, \delta}(\RnRn;M)$ fulfilling \eqref{eq:SlowlyVaryingSymbols} for all $\alpha, \beta \in  \Non$ with $|\alpha| \leq M$ are in the set $\dot{S}^m_{\rho, \delta}(\RnRn;M)$. \\
	A symbol $a \in  S^m_{\rho, \delta} (\RnRn; M)$ belongs to the symbol-class $\tilde{S}^m_{\rho, \delta} (\RnRn; M)$, if for all $\beta \in \Non$ with $|\beta | \neq 0$ we have
  \begin{align*}
    D_x^{\beta} a(x, \xi) \in \dot{S}^{m+\delta|\beta|}_{\rho, \delta} (\RnRn; M).
  \end{align*}
  We call the elements of  $C^{\tilde{m}, \tau} \tilde{S}^m_{\rho, \delta} (\RnRn; M)$ and of $\tilde{S}^m_{\rho, \delta}(\RnRn;M)$ \textbf{slowly varying symbols}.
\end{Def}

The following results are proven in  cf.\  \cite[Lemma 3.7, Lemma 3.8]{Paper3}:

\begin{lemma}\label{lemma:SymbolSmoothing2}
  Let $0 \leq \delta < \rho \leq 1$, $\tilde{m} \in \N_0$, $0<\tau <1$, $M \in \N \cup \{ \infty \}$, $m \in \R$ and $a \in C^{\tilde{m}, \tau} S^m_{\rho, \delta}(\RnRn; M)$. Moreover let $\gamma \in (\delta, \rho)$. Then we have for all $\beta \in \Non$ with $|\beta| \leq \tilde{m}$: 
  \begin{itemize}
      \item[i)] $D_x^{\beta} a^{\sharp}(x, \xi) \in S^{m+\delta|\beta|}_{\rho, \gamma}(\RnRn; M)$,
      \item[ii)] if $a \in C^{\tilde{m}, \tau} \dot{S}^m_{\rho, \delta}(\RnRn; M)$ or if $|\beta| \neq 0$ and $a \in C^{\tilde{m}, \tau} \tilde{S}^m_{\rho, \delta}(\RnRn; M)$, then $D_x^{\beta} a^{\sharp}(x, \xi) \in \dot{S}^{m+\delta|\beta|}_{\rho, \gamma}(\RnRn; M)$
  \end{itemize}
\end{lemma}

\begin{lemma}\label{lemma:SymbolSmoothing1}
  Let $0 \leq \delta < \rho \leq 1$, $\tilde{m} \in \N_0$, $0<\tau <1$, $M \in \N \cup \{ \infty \}$, $m \in \R$ and $a \in C^{\tilde{m}, \tau} \tilde{S}^m_{\rho, \delta}(\RnRn; M)$ such that
  \begin{align*}
    a(x, \xi) \xrightarrow{|x| \rightarrow \infty} a(\infty, \xi) \qquad \text{for all } \xi \in \Rn.
  \end{align*}
  Moreover we set $b(x, \xi):= a(x, \xi)-a(\infty, \xi)$ for all $x, \xi \in \Rn$. Additionally we define $a^{\sharp}, a^b, a^{\sharp}(\infty, .)$ and $a^b(\infty, .)$ as in Definition \ref{Def:SymbolSmoothing}. Then we have for $\gamma \in (\delta, \rho)$ and $\tilde{\e} \in \left( 0, (\gamma- \delta)\tau \right)$: 
  \begin{itemize}
    \item[i)] $a^{\sharp}(\infty, \xi)= a(\infty, \xi) \in S^m_{\rho, \delta}(\RnRn; 0)$,
    \item[ii)] $a^b(\infty, \xi) = 0$ for all $\xi \in \Rn$,
    \item[iii)] $a^b(x, \xi) \in C^{\tilde{m}, \tau} \tilde{S}_{\rho, \gamma}^{m-(\gamma-\delta)(\tilde{m}+\tau)+\tilde{\e}}(\RnRn; M) 
		  \cap C^{\tilde{m}, \tau} \dot{S}_{\rho, \gamma}^{m-(\gamma-\delta)(\tilde{m}+\tau)+\tilde{\e}}(\RnRn; 0) $,
    \item[iv)] $a^{\sharp}(x, \xi) = a(\infty, \xi) + b^{\sharp}(x, \xi)$ for all $x, \xi \in \Rn$.
  \end{itemize}
\end{lemma}

In order to prove the invariance of the Fredholm index we also need the next two statements:

\begin{kor}\label{kor:Eigenschaft2}
  Let $\tilde{m}_1,l \in \N$, $0 < \tau_1 <1$, $m_1, m_2 \in \R$, $0 \leq \delta < \rho \leq 1$; $M_1, M_2 \in \N_0 \cup \{ \infty \}$ with $M_1>n+1$. Additionally let $N:= M_1-(n+1)$. For $a_1 \in C^{\tilde{m}_1, \tau_1} S^{m_1}_{\rho, \delta}(\RnRn; M_1; \mathscr{L}(\C^l))$ and $a_2 \in S^{m_2}_{\rho, \delta}(\RnRn; M_2; \mathscr{L}(\C^l))$ we define for all $x,\xi \in \Rn$ the symbol $a(x,\xi):= (a_{i,j}(x,\xi))_{i,j=1}^l$ via
  \begin{align*}
    a_{i,j}(x, \xi):= \osint e^{-iy\cdot \eta} \left(a_1(x, \xi+\eta) a_2(x+y, \xi) \right)_{i,j} dy \dq \eta \qquad \text{for all } i,j=1, \ldots, l
  \end{align*}
  and for all $k \in \N$ with $k \leq N$, $\gamma \in \Non$ with $|\gamma| \leq N$ and $\theta \in[0,1]$ we set
  \begin{itemize}
    \item $a_1\sharp_k a_2 (x, \xi):= \sum\limits_{|\gamma| <k} \frac{1}{\gamma !} \pa{\gamma} a_1(x,\xi) D_x^{\gamma} a_2(x,\xi)$,
    \item $r_{\gamma, \theta}^{i,j}(x, \xi):= \osint  e^{-iy\cdot \eta} \left( \p_{\eta}^{\gamma}  a_1(x, \xi+\theta \eta) D_y^{\gamma} a_2( x+y,\xi) \right)_{i,j} dy \dq \eta$
  \end{itemize}
  for all $x, \xi \in\Rn$ and $i,j=1, \ldots, l$. Moreover we define $R_k:= (R_k^{i,j})_{i,j=1}^l: \RnRn \rightarrow \mathscr{L}(\C^l)$ by 
  \begin{align*}
    R_k^{i,j}(x,\xi):= k  \sum_{|\gamma|=k} \int_0^1\frac{(1-\theta)^{k-1}}{\gamma!}r^{i,j}_{\gamma, \theta}(x,\xi) d\theta ,
  \end{align*}
  for all $x,\xi \in \Rn$. Then 
  $$a(x, \xi) = a_1\sharp_k a_2 (x, \xi) + R_k(x, \xi) \qquad \text{for all } x, \xi \in \Rn$$
  and with $\tilde{M}_k:= \min \{ M_1-k+1; M_2\}$ and $\tilde{N}_k:= \min \{M_1-k-(n+1); M_2\}$ we obtain 
  \begin{itemize}
    \item $a_1\sharp_k a_2 (x, \xi) \in C^{\tilde{m}_1, \tau_1} S^{m_1+m_2}_{\rho, \delta}(\RnRn; \tilde{M}_k; \mathscr{L}(\C^l))$,
    \item $R_k (x, \xi) \in C^{\tilde{m}_1, \tau_1} S^{m_1+m_2-(\rho-\delta)k}_{\rho, \delta}(\RnRn; \tilde{N}_k; \mathscr{L}(\C^l))$.
  \end{itemize}
  In particular we have $a \in  C^{\tilde{m}_1, \tau_1} S^{m_1+m_2}_{\rho, \delta}(\RnRn; \tilde{N}_1; \mathscr{L}(\C^l))$. If we even have $a_2 \in \tilde{S}^{m_2}_{\rho, \delta}(\RnRn; M_2; \mathscr{L}(\C^l))$, then  $R_k \in C^{\tilde{m}_1, \tau_1} \dot{S}^{m_1+m_2-(\rho-\delta)k}_{\rho, \delta}(\RnRn; \tilde{N}_k; \mathscr{L}(\C^l))$ for all  $k \in \N$ with $k \leq N$. 
\end{kor}

For $l=1$ the previous Corollary was proved in  \cite[Corollary 4.6]{Paper3}. The general case can be proved in the same way, taking into account, that in each step of the proof we can use case $l=1$ due to the definition of an matrix product. Choosing $\tilde{\e}$ small enough in \cite[Theorem 4.7]{Paper3} provides  
the statement of the next result for $p>1$. Since the compactness result for non-smooth pseudodifferential operators of Marschall in \cite[Theorem 4]{MarschallOnTheBoundednessAndCompactness} holds for $p=1$, too, Theorem  4.7 of \cite{Paper3} is also true for $p=1$. The proof is exactly the same as in the case $p>1$.

\begin{thm}\label{thm:Eigenschaften1}
  Let $\tilde{m}_1\in \N_0$, $\tilde{m}_2 \in \N$, $0<\tau_1, \tau_2 <1$, $0\leq \delta <\rho \leq 1$ and $m_1,m_2 \in \R$. Moreover let $p=2$ if $\rho \neq 1$ and $1 \leq p<\infty$ else. We choose  $\theta \notin \N_0$ with $\theta \in  \left(0,(\tilde{m}_2 + \tau_2)(\rho-\delta) \right)$ 
and define $(\tilde{m},\tau):= (\lfloor s \rfloor, s-\lfloor s \rfloor)$, where $s:=\min\{ \tilde{m}_1 + \tau_1; \tilde{m}_2+ \tau_2- \lfloor \theta \rfloor\}$. Additionally let $M_1, M_2 \in \N_0 \cup\{\infty\}$ with $M_1>(n+1)+\lceil \theta \rceil + n\cdot \max\{ \frac{1}{2}, \frac{1}{p} \}$ and $M_2> n\cdot \max\{ \frac{1}{2}, \frac{1}{p} \}$.
  Moreover let $a_1 \in C^{\tilde{m}_1,\tau_1} S^{m_1}_{\rho,\delta}(\RnRn;M_1)$ and $a_2 \in C^{\tilde{m}_2,\tau_2} \tilde{S}^{m_2}_{\rho,\delta}(\RnRn;M_2)$ such that
  \begin{align*}
    a_2(x, \xi) \xrightarrow{|x|\rightarrow \infty} a_2(\infty, \xi) \qquad \text{for all } \xi \in \Rn.
  \end{align*}
  Then we get for each $s \in \R$ fulfilling $(1-\rho)\frac{n}{p}-(1-\delta)(\tilde{m}_2+\tau_2)+\theta < s+m_1 < \tilde{m} + \tau_2$ and $(1-\rho)\frac{n}{p}-(1-\delta)(\tilde{m}+\tau)+\frac{\tilde{m}+\tau}{\tilde{m}_2+\tau_2}\cdot \theta <s < \tilde{m}+\tau$, that
  \begin{align*}
    a_1(x, D_x)a_2(x, D_x)-(a_1 \sharp_{\lceil \theta \rceil} a_2)(x, D_x): H_p^{s+m_1+m_2}(\Rn) \rightarrow H^s_p(\Rn) \quad \text{is compact.}
  \end{align*}
  where $a_1 \sharp_{\lceil \theta \rceil}a_2(x,\xi)$ is defined as in Corollary \ref{kor:Eigenschaft2}.
\end{thm}


\section{Invariance of the Fredholm Index}\label{InvarianceFredholmIndex}

The aim of this section is to prove the invariance of the Fredholm index for non-smooth pseudodifferential operators, cf.\  Theorem \ref{thm:invarianceOfIndex}. As an ingredient for the proof we need the following sufficient condition for non-smooth pseudodifferential operators to be a Fredholm operator: 

\begin{thm}\label{thm:Fredholmproperty}
  Let $\tilde{m}, l \in \N$, $0<\tau <1$, $0\leq \delta <\rho \leq 1$, $m \in \R$, $M \in \N_0 \cup\{\infty\}$  and  $p \in [1,\infty) $ with $p=2$ if $\rho \neq 1$. Moreover let $a \in C^{\tilde{m},\tau}\tilde{S}^{m}_{\rho,\delta}(\RnRn;M;\mathscr{L}(\C^l))$ be a symbol fulfilling the following properties for some $R>0$ and $C_0>0$:
  \begin{itemize}
    \item[1)] $|\det(a(x,\xi))|\<{\xi}^{-ml}\geq C_0$ for all $x,\xi \in \Rn$ with $|x|+|\xi|\geq R$.
    \item[2)] $a(x, \xi) \xrightarrow{|x| \rightarrow \infty} a(\infty, \xi)$ for all $\xi \in \Rn$.
   \end{itemize}
   Then for all  $M \geq (n+2)+n\cdot\max\{1/2, 1/p\}$ and $s\in \R$ with $$(1-\rho)\frac{n}{p}-(1-\delta)(\tilde{m}+\tau)<s<\tilde{m}+\tau$$
  the operator
  \begin{align*}
    a(x,D_x): H^{m+s}_p(\Rn)^{ l} \rightarrow H^s_p(\Rn)^{ l}
  \end{align*}
  is a Fredholm operator.
\end{thm}

The previous theorem follows from \cite[Theorem 1.1]{Paper3} for all $p \neq 1$, if one chooses $\theta$ and $\tilde{\e}$ small enough. Verifying the proof of \cite[Theorem 1.1]{Paper3} provides, that Theorem \ref{thm:Fredholmproperty} also is true for $p=1$.  As an ingredient for the proof, the following lemma was used, cf.\  \cite[Lemma 4.9]{Paper3}: 

\begin{lemma}\label{lemma:HilfslemmaFuerBeweisDerFredholmproperty}
   Let $\tilde{m} \in \N_0$, $l \in \N$, $0<\tau <1$, $0\leq \delta <\rho \leq 1$ and $M \in \N_0 \cup\{\infty\}$. Additionally let $a \in C^{\tilde{m},\tau}\tilde{S}^{0}_{\rho,\delta}(\RnRn;M;\mathscr{L}(\C^l))$ be such that property $1)$ of Theorem \ref{thm:Fredholmproperty} holds. Moreover let $\psi \in C^{\infty}_b(\Rn)$ be such that $\psi(x)=0$ if $|x|\leq 1$ and $\psi(x)=1$ if $|x|\geq 2$. Then $b:\RnRn \rightarrow \C^{l\times l}$ defined by
  \begin{align*}
    b(x,\xi):= \psi(R^{-2}(|x|^2+ |\xi|^2))a(x,\xi)^{-1} \qquad \text{for all } x,\xi \in \Rn
  \end{align*}
  is an element of  $C^{\tilde{m}, \tau} \tilde{S}^0_{\rho, \delta}(\RnRn;M;\mathscr{L}(\C^l))$.
\end{lemma}

We also use the next regularity result for non-smooth pseudodifferential operators in order to prove the invariance of the Fredholm index:

\begin{lemma}\label{lemma:Lemma3.3}
  Let  $l \in \N$, $0\leq \delta < \rho \leq1 $, $m \in \R$,  $\tilde{m} \in \N_0$, $0<\tau <1$, $p \in [1,\infty)$ and $M \in \N_0 \cup \{ \infty \}$ with $M \geq n+ 2$. If  $\rho \neq 1$ we additionally assume $p=2$ and $\tilde{m}+\tau > \frac{1-\rho}{1-\delta} \cdot \frac{n}{2}$. 
  We consider a 
 symbol $a \in C^{\tilde{m}, \tau} \tilde{S}_{\rho,\delta}^m(\RnRn; M; \mathscr{L}(\C^l))$ fulfilling the following properties:
  \begin{itemize}
    \item[1)] $|\det (a(x,\xi))|\<{\xi}^{-ml}\geq C_0$ for all $x,\xi \in \Rn$ with $|x|+|\xi|\geq R$
	\item[2)] $\|\pa{\alpha} a(x, \xi) -  \pa{\alpha}a(\infty, \xi)\|_{\mathcal{L}(\C^l)}\<{\xi}^{-m+\rho |\alpha|} \xrightarrow{|x| \rightarrow \infty} 0$ uniformly in $\xi \in \Rn$ for all $\alpha \in \Non$ with $|\alpha| \leq n+2$.
  \end{itemize}
  If $u \in \left(H^{m+s}_p(\Rn) \right)^l$ with $(1-\rho)\frac{n}{p}-(1-\delta)(\tilde{m}+\tau)<s<\tilde{m}+\tau$ and $a(x, D_x)u = \Phi \in \left( C_c^{\infty}(\Rn) \right)^l$, then $u \in \left( H^{r+m}_q(\Rn) \right)^l$ for all $q \in [1, \infty)$ and $r \in \R$ with $r<\tilde{m}+\tau$.
\end{lemma}

\begin{bem}\label{bem:SymbolClassOfaInfty}
	Since $a=(a_{i,j})_{i,j=1}^l \in  C^{\tilde{m}, \tau} S_{\rho,\delta}^m(\RnRn; M; \mathscr{L}(\C^l))$ we have due to Remark \ref{bem:MatrixwertigerHölderRaum}  and the Leibniz-rule for all $\alpha \in \Non$ with $|\alpha| \leq M-1$  
	\begin{align*}
	&\max_{i,j =1, \ldots, l} \| \pa{\alpha}a_{i,j}(x,\xi) \<{\xi}^{-m+\rho|\alpha|}\|_{C^{0,1}(\Rn_{\xi})}
	\leq \max_{i,j =1, \ldots, l} \| \pa{\alpha}a_{i,j}(x,\xi) \<{\xi}^{-m+\rho|\alpha|}\|_{C^{1}_b(\Rn_{\xi})}  \\
	& \leq \max_{i,j =1, \ldots, l} \max_{k=1, \ldots, n}\left\{ \sup_{\xi \in \Rn} \| \pa{\alpha +e_k}a_{i,j}(x,\xi) \<{\xi}^{-m+\rho|\alpha|}\|_{C_b^{0} (\Rn_{x})} \right. \\
	&\qquad \qquad \qquad \qquad \left.+\sup_{\xi \in \Rn} \| \pa{\alpha}a_{i,j}(x,\xi) \<{\xi}^{-m+\rho|\alpha|}\|_{C_b^{0}(\Rn_{x})}   \right\} 
	 \leq C \qquad \text{ for all } x \in \Rn.
	\end{align*}
	where each entry of $e_k \in \N^n_0$ is $0$ except the $k-$th one, which is $1$.
	Taking the limit $|x| \rightarrow \infty$ on both sides of this inequality yields for all $\xi \in \Rn$:
	\begin{align}\label{eq:rem1}
		|\pa{\alpha}a_{i,j}(\infty, \xi)  \<{\xi}^{-m+\rho|\alpha|}| \leq \|\pa{\alpha}a_{i,j}(\infty, \xi)  \<{\xi}^{-m+\rho|\alpha|}\|_{C^{0,1}(\Rn_{\xi})} 
	\leq C 
	\end{align}
	for all $i,j = 1,\ldots, l$. Consequently we have
	\begin{align*}
		|\pa{\alpha}a(\infty, \xi_1) \<{\xi_1}^{-m+\rho|\alpha|}-\pa{\alpha}a(\infty, \xi_2) \<{\xi_2}^{-m+\rho|\alpha|}| \leq C | \xi_1 -\xi_2 | \xrightarrow{\xi_1 \rightarrow \xi_2} 0
	\end{align*} 
	for all $i,j = 1,\ldots, l$.
	Because of $\<{\xi}^{-m+\rho|\alpha|} \in C^{\infty}(\Rn_{\xi})$ this implies $\pa{\alpha} a(\infty, \xi) \in C^0(\RnRn; \C^{l\times l})$. Together with \eqref{eq:rem1} and Remark \ref{bem:MatrixwertigerHölderRaum}  we obtain 
	\begin{align}\label{eq:SymbolClassOfaInfty}
		a(\infty, \xi) \in S^m_{\rho, \delta}(\RnRn; M-1; \mathscr{L}(\C^{l})).
	\end{align}
\end{bem}

\begin{bem}\label{bem:Regularity1}
  Let us even assume $M \geq n+4$ and for all $\alpha \in \Non$ with $|\alpha| \leq n+3$ we assume $\pa{\alpha} a(x, \xi) \rightarrow \pa{\alpha} a(\infty, \xi)$ for all $\xi \in \Rn$ if $|x|\rightarrow \infty$ in the previous lemma. Then we can exchange assumption 2) with the following weaker on:
  \begin{itemize}
   \item[2')] $\|a(x, \xi) -  a(\infty, \xi)\|_{\mathscr{L}(\C^l)}\<{\xi}^{-m} \xrightarrow{|x| \rightarrow \infty} 0$ uniformly in $\xi \in \Rn$.
  \end{itemize}
  More precisely: Due to Remark \ref{bem:SymbolClassOfaInfty} we have $a(\infty, \xi)=\left(a^{ik}(\infty, \xi) \right)_{i,k=1}^l \in  S_{\rho,\delta}^m(\RnRn; n+3;\mathscr{L}(\C^l))$.  Now let $\alpha \in \Non$ with $|\alpha| \leq n+2$ be arbitrary. Additionally let $(\varphi_j)_{j \in \N_0}$ be a dyadic partition of unity defined as in Section \ref{section:Preliminaries}. 
   Hence there are two constants $C_1,C_2>0$ such that for all $j \in \N$
   \begin{align}\label{ee2}
      C_1 2^j \leq \<{\xi} \leq C_2 2^j \qquad \text{for all } \xi \in \supp{\varphi_j}.
   \end{align}
   Now an application of Lemma \ref{lemma:interpolation1} yields for  $\theta:= \frac{|\alpha|}{|\alpha|+1}$ and $$\tilde{a}_j(x,\xi)=\left(\tilde{a_j}^{ik}(\infty, \xi) \right)_{i,k=1}^l:= \left[ a(x,\xi)-a(\infty, \xi)\right]\varphi_j(\xi), \qquad x,\xi \in \Rn$$ the existence of a $C>0$, independent of $j \in \N_0$, such that:
   \begin{align*}
      |\pa{\alpha} \tilde{a}^{ik}_j(x,\xi)| \cdot 2^{-mj+\rho |\alpha|j} 
      \leq C \|\tilde{a}^{ik}_j(x,.)\|^{1-\theta}_{C_b^0(\Rn)} \left(\max_{|\beta|=|\alpha|+1} \|\pa{\beta}\tilde{a}^{ik}_j(x,.)\|_{C_b^0(\Rn)}\right)^{\theta} \cdot 2^{-mj+ \rho|\alpha|j}
   \end{align*}
   for all $i,k=1, \ldots, l$, $j \in \N_0$. Using $2^{-mj+\rho |\alpha|j}=2^{-mj(1-\theta)} \cdot 2^{(-mj+\rho(|\alpha|+1)j)\theta}$ and (\ref{ee2}) provides
   \begin{align}\label{ee3}
      &|\pa{\alpha} \tilde{a}^{ik}_j(x,\xi)| \<{\xi}^{-m+ \rho|\alpha|}  \notag \\
      &\qquad \qquad  \leq C \|\tilde{a}^{ik}_j(x,.)\<{\xi}^{-m}\|^{1-\theta}_{C_b^0(\Rn)} \left(\max_{|\beta|=|\alpha|+1} \|\<{\xi}^{-m+\rho|\beta|}\pa{\beta}\tilde{a}^{ik}_j(x,.)\|_{C_b^0(\Rn)}\right)^{\theta}
   \end{align}
    for all $i,k=1, \ldots, l$, $j \in \N_0$. On account of \eqref{eq:GlattesSymbolIstNichtglatt} we know that $a(\infty, \xi)$ and $a(x,\xi)=\left(a^{ik}(x, \xi) \right)_{i,k=1}^l$ are elements of $ C^{\tilde{m},\tau} S^m_{\rho, \delta}(\RnRn; n+3; \mathscr{L}(\C^l)).$ Together with $\varphi_0 \in C^{\infty}_c(\Rn)$ and (\ref{ee2}) we can show by means of the Leibniz-rule, that
    \begin{align}\label{ee4}
      |\pa{\beta}  \tilde{a}^{ik}_j(x,\xi)| \leq C_{\beta} \<{\xi}^{m-\rho|\beta|} \quad \text{for all } x,\xi \in \Rn, j \in \N_0, |\beta| \leq n+3
    \end{align}
    for all $i,k =1, \ldots, l$. A combination of (\ref{ee3}),(\ref{ee4}) and $2')$ yields
    \begin{align*}
      \|\pa{\alpha} \tilde{a}_j(x,\xi)\|_{\mathscr{L}(\C^l)} \<{\xi}^{-m+ \rho|\alpha|}  \leq C_{\theta} \sup_{\xi \in \Rn} \| \<{\xi}^{-m} \left[ a(x,\xi)-a(\infty, \xi)\right] \|^{1-\theta}_{\mathscr{L}(\C^l)} \xrightarrow{|x|\rightarrow \infty} 0
    \end{align*}
    uniformly in $\xi \in \Rn$ and $j \in \N_0$. Since for each $\xi \in \Rn$ the right side of the equality 
       $a(x,\xi)-a(\infty, \xi)=\sum_{j=0}^{\infty} \tilde{a}_j(x,\xi)$
    consists of not more than two terms which are not equal to $0$ and since the previous convergence is uniform in $j$ and $\xi$, assumption $2)$ of the previous lemma is true.
\end{bem}

In order to verify Lemma \ref{lemma:Lemma3.3}, we show 
\begin{lemma}\label{HilfslemmaFuerLemma3.3}
  Let $l \in \N$, $\tilde{m} \in \N_0$, $0<\tau <1$, $0\leq \rho, \delta \leq 1$, $\delta<1$, $m \in \R$ and $M \in \N_0 \cup \{ \infty \}$ with $M \geq n+ 1$. Additionally let $0<s<\tau$.
  We consider a  symbol $a=(a_{i,j})_{i,j=1}^l \in C^{\tilde{m}, \tau} \tilde{S}_{\rho,\delta}^m(\RnRn; M; \mathscr{L}(\C^l))$ fulfilling the property
\begin{itemize}
    \item[1)] $\|\pa{\alpha} a(x, \xi) -  \pa{\alpha}a(\infty, \xi)\|_{\mathscr{L}(\C^l)}\<{\xi}^{-m+\rho|\alpha|} \xrightarrow{|x| \rightarrow \infty} 0$ uniformly in $\xi \in \Rn$ for all $\alpha \in \Non$ with $|\alpha| \leq n+1$.
  \end{itemize}
Moreover let $\psi \in C^{\infty}(\Rn)$ with $\psi(x)=1$ for $|x|\geq 2$ and $\psi(x) = 0$ if $|x|\leq 1$. Then for all $\e >0$ there is a constant $\hat{R}>1$ such that
	\begin{align*}
		\hat{a}(x,\xi):= a(x,\xi)\psi(\hat{R}^{-1}x) + a(\infty,\xi)(1-\psi(\hat{R}^{-1}x)) \qquad \text{for all } x, \xi \in \Rn
	\end{align*}
	fulfills for all $k \leq n+1$ 
	\begin{align}\label{eq:EstimateOfSymbolaHat}
		|\hat{a}(x, \xi)-a(\infty, \xi)|_{k, C^{\tilde{m},s}S^m_{\rho, \delta}(\RnRn,M; \mathscr{L}(\C^l))} \leq \e.
	\end{align}
\end{lemma}

\begin{proof}
	First we assume $l=1$. Let $k \leq n+1$, $0<s<\tau$ and $\e >0$ be arbitrary but fixed. Without loss of generality we can assume, that $0\leq \psi \leq 1$.
	Assumption $1)$ provides the existence of an $R_1>1$ such that we have for  all $\alpha \in \Non$ with $|\alpha| \leq k$:
	\begin{align*}
	  |\pa{\alpha}\{a(x,\xi)-a(\infty,\xi)\}| \<{\xi}^{-m+\rho|\alpha|} \leq \frac{\e}{2}  
\qquad \text{ for all } x,\xi \in \Rn \text{ with } |x|\geq R_1.
	\end{align*}
	By means of the previous inequality, the Leibniz rule and $\psi(R_1^{-1}x)=0$ for all $|x|\leq R_1$ we get for all $\hat{R}\geq R_1$ and all  $\alpha \in \Non$ with $|\alpha| \leq k$:
	\begin{align}\label{eq:HelpEstimateOfSymbolaHat1}
	  &\|\pa{\alpha} \hat{a}(.,\xi) - \pa{\alpha} a(\infty, \xi )\|_{C^0_b(\Rn)} \<{\xi}^{-m+\rho|\alpha|} 
	  \leq \frac{\e}{2}  \qquad \text{for all } \xi \in \Rn.
	\end{align}
	A combination of \eqref{eq:HelpEstimateOfSymbolaHat1} and the next estimate yields \eqref{eq:EstimateOfSymbolaHat}:
	\begin{align}\label{eq:HelpEstimateOfSymbolaHat2}
	  \max_{|\alpha| \leq k} \sup_{\xi \in \Rn} \left\{ \|\pa{\alpha} \hat{a}(.,\xi)-\pa{\alpha} a(\infty, \xi)\|_{C^{\tilde{m},s}(\Rn)} \<{\xi}^{-m+\rho|\alpha|-\delta(\tilde{m}+s)}\right\} \leq \frac{\e}{2}.
	\end{align}
	In order to show \eqref{eq:HelpEstimateOfSymbolaHat2}  let $\alpha, \beta \in \Non$ with $|\alpha|\leq k$ and $|\beta| \leq \tilde{m}$ and $R_2>1$ be arbitrary. 
	On account of $\<{\xi}^{-\delta(\tilde{m}+s)(1-|\beta|/\tilde{m})}\leq 1$ we have
	\begin{align*}
	  &\|\{\pa{\alpha} a(x,\xi) -\pa{\alpha}a(\infty, \xi)\}\psi(R_2^{-1}x)\|^{|\beta|/\tilde{m}}_{C^{\tilde{m}}_b(\Rn \backslash B_{R_2}(0))} \<{\xi}^{(-m+\rho|\alpha|)\cdot |\beta|/\tilde{m}-\delta(\tilde{m}+s)}\\
	  &\qquad \leq \left( \|\{\pa{\alpha} a(x,\xi) -\pa{\alpha}a(\infty, \xi)\}\psi(R_2^{-1}x)\|_{C^{\tilde{m}}_b(\Rn \backslash B_{R_2}(0))} \<{\xi}^{-m+\rho|\alpha|-\delta(\tilde{m}+s)}  \right)^{|\beta|/\tilde{m}}.
	\end{align*}
	Now we apply Lemma \ref{lemma:PropertyHoelderSpaces} to the previous inequality first and use the properties of the function $\psi$ afterwards. Then we get the existence of a constant $C_{\tilde{m}}$, independent of $R_2>1$, such that 
	\begin{align*}
	  &\|\{\pa{\alpha} a(x,\xi) -\pa{\alpha}a(\infty, \xi)\}\psi(R_2^{-1}x)\|^{|\beta|/\tilde{m}}_{C^{\tilde{m}}_b(\Rn_x \backslash B_{R_2}(0))} \<{\xi}^{(-m+\rho|\alpha|)\cdot |\beta|/\tilde{m}-\delta(\tilde{m}+s)}\\
	  &\qquad \leq C_{\tilde{m}} \left( \|\pa{\alpha} a(x,\xi) -\pa{\alpha}a(\infty, \xi)\|_{C^{\tilde{m},s}(\Rn_x \backslash B_{R_2}(0))} \<{\xi}^{-m+\rho|\alpha|-\delta(\tilde{m}+s)} \right)^{|\beta|/\tilde{m}}.
	\end{align*}
	Since the restriction of functions defined on  $\Rn$ to the domain $\Rn \backslash B_{R_2}(0))$ is continuous and since $a(x,\xi), a(\infty, \xi) \in C^{\tilde{m},\tau}S^m_{\rho, \delta}(\RnRn; n+1) \subseteq C^{\tilde{m},s}S^m_{\rho, \delta}(\RnRn; n+1)$ we obtain for all $\xi \in \Rn$
	\begin{align}\label{eq:HelpEstimateOfSymbolaHat3}
	  \|\{\pa{\alpha} a(x,\xi) -\pa{\alpha}a(\infty, \xi)\}\psi(R_2^{-1}x)\|^{|\beta|/\tilde{m}}_{C^{\tilde{m}}_b(\Rn_x \backslash B_{R_2}(0))} \<{\xi}^{(-m+\rho|\alpha|)\cdot |\beta|/\tilde{m}-\delta(\tilde{m}+s)} \leq C_{\tilde{m}},
	\end{align}
	where $C_{\tilde{m}}$ is independent of $R_2>1$ and $\xi \in \Rn$. Let $\theta_{\beta}:= \frac{|\beta|+s}{\tilde{m}+\tau}$.
	Similarly to \eqref{eq:HelpEstimateOfSymbolaHat3} we can show the existence of a constant $B_{\tilde{m}}$, independent of $R_2>1$, such that we have for all $\xi \in \Rn$
	\begin{align}\label{eq:HelpEstimateOfSymbolaHat4}
	  \|\{\pa{\alpha} a(x,\xi) -\pa{\alpha}a(\infty, \xi)\}\psi(R_2^{-1}x)\|^{\theta_{\beta}}_{C^{\tilde{m}, \tau}(\Rn_x \backslash B_{R_2}(0))} \<{\xi}^{(-m+\rho|\alpha|)\cdot \theta_{\beta}-\delta(\tilde{m}+s)} \leq B_{\tilde{m}}.
	\end{align}
Now let $C_{|\beta|, \tilde{m}}$ and $C_{\theta_{\beta}}$ be the constants of Lemma \ref{lemma:interpolation3}.
	Due to $\psi \in C^{\infty}_b(\Rn)$ and assumption $1)$ there is a $R_3>1$ such that we get for all $\hat{R}\geq R_3$
	\begin{align}\label{eq:HelpEstimateOfSymbolaHat5}
	  &\|\{\pa{\alpha} a(x,\xi) -\pa{\alpha}a(\infty, \xi)\}\psi(\hat{R}^{-1}x)\|^{1-|\beta|/\tilde{m}}_{C^0_b(\Rn_x \backslash B_{\hat{R}}(0))} \<{\xi}^{(-m+\rho|\alpha|)(1-|\beta|/\tilde{m})} \leq \frac{\e}{4C_{\tilde{m}}\cdot C_{|\beta|, \tilde{m}}  } \\ \label{eq:HelpEstimateOfSymbolaHat6}
	&\hspace{-5mm}\text{and} \notag\\
	  &\|\{\pa{\alpha} a(x,\xi) -\pa{\alpha}a(\infty, \xi)\}\psi(\hat{R}^{-1}x)\|^{1-\theta_{\beta}}_{C^0_b(\Rn_x \backslash B_{\hat{R}}(0))} \<{\xi}^{(-m+\rho|\alpha|)(1-\theta_{\beta})} \leq \frac{\e}{4B_{\tilde{m}}
	\cdot C_{\theta_{\beta}}
}.
	\end{align}
	Now we choose $\hat{R}:=\max\{R_2, R_3\}$. Since $\psi(\hat{R}^{-1}x)=0$ for all $|x| \leq \hat{R}$, we have
	\begin{align*}
	  &\|\pa{\alpha} \hat{a}(.,\xi) -\pa{\alpha}a(\infty, \xi)\|_{C^{\tilde{m},s}(\Rn)} \<{\xi}^{-m+\rho|\alpha|-\delta(\tilde{m}+s)}\\
	  &\qquad = \|\{ \pa{\alpha} a(.,\xi) -\pa{\alpha}a(\infty, \xi)\} \psi(\hat{R}^{-1}x)\|_{C^{\tilde{m},s}(\Rn)} \<{\xi}^{-m+\rho|\alpha|-\delta(\tilde{m}+s)}\\
	  &\qquad = \max_{|\beta| \leq \tilde{m}} \left\{ \|D_x^{\beta} \{( \pa{\alpha} a(.,\xi) -\pa{\alpha}a(\infty, \xi)) \psi(\hat{R}^{-1}x)\}\|_{C^{0,s}(\Rn\backslash B_{\hat{R}}(0))} \<{\xi}^{-m+\rho|\alpha|-\delta(\tilde{m}+s)} \right\}.
	\end{align*}
	If we now apply Lemma \ref{lemma:interpolation3} on the previous equality first and use the estimates \eqref{eq:HelpEstimateOfSymbolaHat3}-\eqref{eq:HelpEstimateOfSymbolaHat6} with $R_2:=\hat{R}$ afterwards, we obtain \eqref{eq:HelpEstimateOfSymbolaHat2} for $l=1$.
	We get the general case by applying case $l=1$ to each entry of $\hat{a}$.
\end{proof}

Now we have all results at hand in order to show the regularity result of Lemma \ref{lemma:Lemma3.3}. By means of symbol-smoothing and the continuity results for non-smooth pseudodifferential operators we gain some regularity of small order. On account of this auxiliary regularity result and the embeddings of Bessel potential spaces, we prove the lemma with the aid of a bootstrap argument:

\begin{proof}[Proof of Lemma \ref{lemma:Lemma3.3}:]
	Let $(1-\rho)\frac{n}{p}-(1-\delta)(\tilde{m}+\tau)<s<\tilde{m}+\tau$ and $r \in \R$ with $r<\tilde{m}+\tau$ be arbitrary.  We choose an arbitrary but fixed $\psi \in C^{\infty}(\Rn)$ with $\psi(x)=1$ for $|x| \geq 2$ and $\psi(x)=0$ if $|x|\leq 1$. We divide the proof of this lemma into four different steps. In step one we prove an auxiliary regularity result, needed for the proof of the claim. Afterwards we start with showing the claim by considering different cases. First we treat the case $p=q$ in step two. In the next two steps we consider an arbitrary $p \in [1,\infty)$ and $\rho=1$. In step three we show the claim for $q \in [1,p]$. Finally we treat case $q \in [p,\infty)$ in step 4. 

 \textbf{Step 1. } 
Let $s \leq t < \tilde{m}+\tau$. Additionally let $q_0=2$ if $\rho \neq 1$ and $q_0 \in [1,\infty)$ else. We choose $\delta < \gamma < \rho$ such that 
	\begin{align*}
-(1-\delta)(\tilde{m}+\tau) < 
-(1-\gamma)(\tilde{m}+\tau) <s
	\end{align*} 
	and  define $b:\RnRn \rightarrow \C$ via
  \begin{align*}
    b(x, \xi):= \psi\left( R^{-1}(|x|^2 + |\xi|^2) \right) \tilde{a}(x, \xi)^{-1} \qquad \text{for all }x,\xi \in \Rn,
  \end{align*}
  where $\tilde{a}(x, \xi):= a(x, \xi)\<{\xi}^{-m} \in C^{\tilde{m}, \tau}\tilde{S}^0_{\rho, \delta}(\RnRn; M; \mathscr{L}(\C^l))$. On account of Lemma \ref{lemma:HilfslemmaFuerBeweisDerFredholmproperty} we already know, that $b \in C^{\tilde{m}, \tau}\tilde{S}^0_{\rho, \delta}(\RnRn; M; \mathscr{L}(\C^l))$. 
  Additionally we get due to the Leibniz rule and Lemma \ref{lemma:PropertyHoelderSpaces}, that 
  \begin{align}\label{62e}
    b(x, \xi)\tilde{a}(x,\xi)-\mathbbm{1} \in C^{\tilde{m}, \tau}S^0_{\rho, \delta}(\RnRn; M; \mathscr{L}(\C^l)),
  \end{align}
  where $\mathbbm{1} \in \C^{l\times l}$ is the unit-matrix.  Since $b(x, \xi) \tilde{a}(x,\xi)-\mathbbm{1}=0$ for all $x,\xi \in \Rn$ with $|x|^2 + |\xi|^2 \geq 2R$, a straight forward calculation even provides using (\ref{62e}) and $\<{\xi} \leq C_R$ for all $|\xi| \leq 2R$
  \begin{align}\label{63e}
    b(x, \xi) \tilde{a}(x,\xi)-\mathbbm{1} &\in C^{\tilde{m}, \tau}S^{-\infty}_{\rho, \delta}(\RnRn; M; \mathscr{L}(\C^l))\notag \\
	&\subseteq  C^{\tilde{m}, \tau}S^{-\infty}_{\rho, \gamma}(\RnRn; M; \mathscr{L}(\C^l)).
  \end{align}
  By means of the symbol-smoothing  we can split the symbol $\tilde{a}$ in two symbols $\tilde{a}^{\sharp}:=(a^{\sharp}_{i,j})_{i,j=1}^l$ and $\tilde{a}^b:=(a^{b}_{i,j})_{i,j=1}^l$ defined as in Subsection \ref{Subsection:SymbolSmoothing}. Due to Lemma \ref{lemma:SymbolSmoothing1} and Lemma \ref{lemma:SymbolSmoothing2} we then have   
  \begin{itemize}
    \item[i)] $\tilde{a}^b \in C^{\tilde{m}, \tau} \tilde{S}_{\rho, \gamma}^{-\theta}(\RnRn, M; \mathscr{L}(\C^l)) \cap C^{\tilde{m}, \tau} \dot{S}_{\rho, \gamma}^{-\theta}(\RnRn; 0; \mathscr{L}(\C^l)) $,
    \item[ii)] $\tilde{a}^{\sharp} \in \tilde{S}^{0}_{\rho, \gamma}(\RnRn, M; \mathscr{L}(\C^l))$.
  \end{itemize}
for some $0< \theta < 1$. Since $\tilde{a}^{\sharp}$ is  smooth with respect to the first variable, the composition $b(x, D_x) \tilde{a}^{\sharp}(x, D_x)$ is a pseudodifferential operator again with symbol $b \sharp \tilde{a}^{\sharp}$.
Then Corollary \ref{kor:Eigenschaft2} provides that the symbol $b \sharp_{\lceil \theta \rceil} \tilde{a}^{\sharp} $ 
  has the following property:
  \begin{itemize}
    \item[iii)] $b \sharp \tilde{a}^{\sharp} -  b \sharp_{\lceil \theta \rceil} \tilde{a}^{\sharp} \in  C^{\tilde{m}, \tau} \dot{S}^{-(\rho-\gamma)}_{\rho, \gamma}(\RnRn; M-(n+2);  \mathscr{L}(\C^l))$.
  \end{itemize} 
  Using $\tilde{a}=\tilde{a}^b+\tilde{a}^{\sharp}$  and $iii)$ provides
  \begin{align}\label{eq38ee}
    &b(x,D_x) \tilde{a}(x,D_x) - \left( b \sharp_{\lceil \theta \rceil} \tilde{a} \right)(x,D_x) \notag\\ 
    &\qquad = b(x,D_x)\tilde{a}^b(x,D_x) + \left( b \sharp_{\lceil \theta \rceil} \tilde{a}^{\sharp} \right)(x,D_x)  - \left( b \sharp_{\lceil \theta \rceil} \tilde{a} \right)(x,D_x) + R_{\lceil \theta \rceil }(x, D_x) \notag \\
    &\qquad =b(x,D_x)\tilde{a}^b(x,D_x) - ( b  \tilde{a}^b)(x,D_x)  + R_{\lceil \theta \rceil }(x, D_x),
  \end{align}
  where
  \begin{align}\label{61ee}
    R_{\lceil \theta \rceil } \in C^{\tilde{m}, \tau} \dot{S}_{\rho, \gamma}^{-(\rho-\gamma)} (\RnRn; M-(n+2); \mathscr{L}(\C^l)).
  \end{align}
  By means of the Leibniz rule,  Lemma \ref{lemma:PropertyHoelderSpaces} and $ \tilde{a}^b \in C^{\tilde{m}, \tau} \tilde{S}_{\rho, \gamma}^{-\theta}(\RnRn; M; \mathscr{L}(\C^l))$ we obtain:
  \begin{align}\label{60ee}
     b(x, \xi) \tilde{a}^b(x,\xi) \in C^{\tilde{m}, \tau} S^{-\theta}_{\rho, \gamma}(\RnRn; M; \mathscr{L}(\C^l)).
  \end{align}
  An application of Theorem \ref{thm:BoundednessResultNonSmooth} yields because of $i)$, $b \in C^{\tilde{m}, \tau}\tilde{S}^0_{\rho, \gamma}(\RnRn; M; \mathscr{L}(\C^l))$, (\ref{60ee}), (\ref{61ee}), (\ref{eq38ee}) and (\ref{63e}) for  $m_n:=\min\{\theta; \rho-\gamma; (\tilde{m}+\tau-t)/2 \}$:
  \begin{itemize}
    \item[iv)] $R:= b(x,D_x) \tilde{a}(x,D_x) - \left( b \sharp_{\lceil \theta \rceil} \tilde{a} \right)(x,D_x): \left( H^t_{q_0}(\Rn) \right)^l \rightarrow \left( H^{t+m_n}_{q_0}(\Rn) \right)^l$,
    \item[v)] $\op\left[ b(x, \xi) \tilde{a}(x, \xi) -\mathbbm{1} \right]: \left(H^t_{q_0}(\Rn) \right)^l \rightarrow \left( H^{t+m_n}_{q_0}(\Rn) \right)^l$.
  \end{itemize}
  Moreover we obtain for the parametrix $Q:= \<{D_x}^{-m} b(x, D_x)$ of $a(x, D_x)$:
  \begin{align}\label{65e}
    Q a(x, D_x)
    &= \<{D_x}^{-m} b(x, D_x) \tilde{a}(x, D_x) \<{D_x}^{m} \notag \\
    &= \id + \<{D_x}^{-m} \op\left[  b(x, \xi) \tilde{a}(x, \xi) -\mathbbm{1}  \right]  \<{D_x}^{m} + \<{D_x}^{-m} R \<{D_x}^{m}.
  \end{align}
  Now let $u \in \left( H^{m+t}_{q_0}(\Rn) \right)^l$ with $a(x, D_x)u = \Phi \in \left( C^{\infty}_c(\Rn) \right)^l \subseteq \left( H^{t+m_n}_{q_0}(\Rn) \right)^l$ be arbitrary. By means of the continuity results for pseudodifferential operators, cf.\  Theorem \ref{thm:BoundednessResultNonSmooth} and Theorem \ref{thm:ConituityResultSmoothCase}, we obtain $Q a(x, D_x)u = Q \Phi \in \left( H^{m+m_n+t}_{q_0}(\Rn) \right)^l$. Together with $iv)$ and $v)$ and equality (\ref{65e}) we therefore get that
  \begin{align*}
    u = Qa(x, D_x)u - \<{D_x}^{-m} \op\left[  b(x, \xi) \tilde{a}(x, \xi) -\mathbbm{1}  \right]  \<{D_x}^{m}u - \<{D_x}^{-m} R \<{D_x}^{m}u 
  \end{align*}
  is an element of $\left( H^{m+m_n+t}_{q_0}(\Rn) \right)^l$. In summary, this provides
\begin{align}\label{eq:FootStep}
	\text{if } u \in \left( H^{m+t}_{q_0}(\Rn) \right)^l \text{ and } a(x, D_x)u  \in \left( C^{\infty}_c(\Rn) \right)^l,\text{ then } u \in \left( H^{m+m_n+t}_{q_0}(\Rn) \right)^l.
\end{align}	

\textbf{Step 2.}
In this step we prove the claim for $p=q$. If $r\leq s$ we immediately obtain $u \in \left( H^{m+s}_p(\Rn) \right)^l \subseteq \left( H^{m+r}_p(\Rn) \right)^l$. Hence it remains to consider $r \in \R$ with $s<r<\tilde{m}+\tau$. To this end we consider a  $u \in \left( H^{m+s}_p(\Rn) \right)^l$ such that $ a(x, D_x)u  \in \left( C^{\infty}_c(\Rn) \right)^l$. Due to \eqref{eq:FootStep} for $t=s$ and $q_0=p$ we obtain $u \in \left( H^{m+s_1}_p(\Rn) \right)^l$ where $s_1:= s + \min\{\theta; \rho-\gamma; (\tilde{m}+\tau-s)/2 \}$. Now we distinguish two different cases. If $s_1<r$, we repeat the application of \eqref{eq:FootStep} applied on  $t=s_1$ and $q_0=p$ and get $u \in \left( H^{m+s_2}_p(\Rn) \right)^l$ where $s_2:= s_1 + \min\{\theta; \rho-\gamma; (\tilde{m}+\tau-s_1)/2 \}$. We repeat this step until we obtain $s_i \geq r$ for some $i \in \N$. This is possible since $r<\tilde{m}+\tau$. If $s_1 \geq r $ instead, we obtain the claim, since $u \in \left( H^{m+s_1}_p(\Rn) \right)^l \subseteq \left( H^{m+r}_p(\Rn) \right)^l$.  

\textbf{Step 3.} Now we assume $\rho=1$. The aim of this step is to show the claim for $q \in [1,p]$ and arbitrary $r$ fulfilling the assumptions.  Without loss of generality we can assume  on account of Step 2  that $r=s$.  
Due to Remark \ref{bem:SymbolClassOfaInfty} we have $a(\infty, \xi) \in  S_{\rho,\delta}^m(\RnRn; n+1; \mathscr{L}(\C^l))$. 
  On account of \eqref{eq:GlattesSymbolIstNichtglatt} and Theorem \ref{thm:BoundednessResultNonSmooth} we obtain the boundedness of 
  \begin{align}\label{ee5}
    a(\infty, D_x):\left(H^{m+s}_{p_1}(\Rn)\right)^l \rightarrow \left( H^{s}_{p_1}(\Rn) \right)^l \qquad \text{for all } p_1 \in [1,\infty).
  \end{align}
  Assumption $1)$ implies $|\det \left( a(\infty,\xi) \right)| \neq 0$ for all $\xi \in \Rn$. Hence $a(\infty, D_x)$ is invertible with inverse $b(D_x)$, where
  \begin{align}\label{eq:InverseOfaInfty}
    b(\xi):=[a(\infty,\xi)]^{-1} \in S^{-m}_{\rho, \delta}(\RnRn;0; \mathscr{L}(\C^l))
  \end{align}
  follows immediately of $1)$ for $l=1$.  For general $l \in \N$ the property \eqref{eq:InverseOfaInfty} follows from Cramer's rule, case $l=1$ and the fact, that $C^{\tilde{m}, \tau}\tilde{S}^0_{\rho, \delta}(\RnRn; M)$ is closed with respect to pointwise multiplication. Consequently (\ref{ee5}) provides 
  \begin{align*}
    b(D_x) \in \mathscr{L}(\left(H^{s}_{p_1}(\Rn)\right)^l; \left(H^{m+s}_{p_1}(\Rn) \right)^l ) \qquad \text{for all } p_1 \in [1,\infty).
  \end{align*}
  Since the set of all invertible operators in $\mathscr{L}(\left(H^{m+s}_{p_1}(\Rn) \right)^l; \left(H^{s}_{p_1}(\Rn) \right)^l)$ with $p_1 \in \{p,q\}$ is open, there is an $\tilde{R}>0$ such that all $T_{p_1} \in \mathscr{L}(\left(H^{m+s}_{p_1}(\Rn) \right)^l ;\left( H^{s}_{p_1}(\Rn) \right)^l )$ with $p_1 \in \{p,q\}$ fulfilling
  \begin{align}\label{eq:hilfsformelInvertierbarkeit1}
   \|a(\infty, D_x)-T_{p_1}\|_{ \mathscr{L}(\left( H^{m+s}_{p_1}(\Rn) \right)^l ;\left( H^{s}_{p_1}(\Rn) \right)^l)} \leq \tilde{R} \qquad \text{for all } p_1 \in \{p,q\}
  \end{align}
  are invertible in $\mathscr{L}(\left( H^{m+s}_{p_1}(\Rn) \right)^l; \left( H^s_{p_1}(\Rn) \right)^l)$ with  $p_1 \in \{p,q\}$.  Due to  $a(\infty, \xi) \in S^m_{\rho, \delta}(\RnRn;n+1; \mathscr{L}(\C^l))$ and Theorem \ref{thm:BoundednessResultNonSmooth}  inequality \eqref{eq:hilfsformelInvertierbarkeit1} holds for $T_q=T_p=\hat{a}(x,D_x)$ with $\hat{a} \in C^{\tilde{m}, \tilde{\tau}}S^m_{\rho, \delta}(\RnRn; M; \mathscr{L}(\C^l))$, $\tilde{\tau}>0$, fulfilling
\begin{align}\label{eq1Lemma33}
	|\hat{a}-a(\infty, \xi)|_{k,  C^{\tilde{m}, \tilde{\tau}}S^m_{1, \delta}(\RnRn, M; \mathscr{L}(\C^l))}<\e
\end{align}
	for some $\e>0$.
	  On account of Lemma \ref{HilfslemmaFuerLemma3.3} there is an $\hat{R}>1$ such that
	\begin{align*}
		\hat{a}(x,\xi):= a(x,\xi)\psi(\hat{R}^{-1}x) + (1-\psi(\hat{R}^{-1}x)) a(\infty, \xi) \qquad \text{for all } x, \xi \in \Rn
	\end{align*} 
	fulfills \eqref{eq1Lemma33} for $\tilde{\tau}>0$ with $s< \tilde{\tau}+\tilde{m}<\tau +\tilde{m}$. Hence  inequality \eqref{eq:hilfsformelInvertierbarkeit1} holds for $T_q=T_p=\hat{a}(x,D_x)$. Consequently   $\hat{a}(x,D_x)$ is invertible as a map from $\left(H^{m+s}_q(\Rn) \right)^l$ to $\left( H^s_q(\Rn) \right)^l$. Now we choose $u \in \left( H^{m+s}_p(\Rn) \right)^l$ such that $a(x,D_x)u=\Phi \in \left( C^{\infty}_c(\Rn) \right)^l$. The continuity of  a pseudodifferential operator with symbol in $ S^m_{\rho,\delta}(\RnRn;n+1; \mathscr{L}(\C^l))$ and $1-\psi(\hat{R}^{-1}x) \in C^{\infty}_c(\Rn)$ with support in $\overline{B_{2\hat{R}}(0)}$ yields 
	\begin{align}\label{eq6:Lemma3.3}
		 (1-\psi(\hat{R}^{-1}x)) a(\infty, D_x)u \in \left( H^s_p(\overline{B_{2\hat{R}}(0)}) \right)^l \subseteq \left( H^s_q(\overline{B_{2\hat{R}}(0)}) \right)^l.
	\end{align}
 By means of
\eqref{eq6:Lemma3.3} we get using $a(x,D_x)u=\Phi $:
\begin{align*}
	\hat{a}(x,D_x)u= \Phi - (1-\psi)(\hat{R}^{-1}x) a(x,D_x) u + (1-\psi(\hat{R}^{-1}x))a(\infty, D_x)u \in \left( H^s_q(\Rn) \right)^l.
\end{align*} 
Since $\hat{a}(x,D_x) $ is invertible as a map from $\left( H^{m+s}_q(\Rn) \right)^l$ to $\left( H^s_q(\Rn) \right)^l$ and $\hat{a}(x,D_x)u \in \left( H^s_q(\Rn) \right)^l$, we  obtain $u \in \left( H^{m+s}_q(\Rn) \right)^l$.

\textbf{Step 4.} We again assume $\rho =1$. In this step we proof the claim for all $q \in [p,\infty)$ and arbitrary $r$ fulfilling the assumptions. Due to Step 2 we can assume without loss of generality that $r=s$. We define $m_n$ as in Step 2 for $t=s$.  To this end we consider a  $u \in \left( H^{m+s}_p(\Rn) \right)^l$ such that $ a(x, D_x)u  \in \left( C^{\infty}_c(\Rn) \right)^l$. An application of \eqref{eq:FootStep} for $t=s$ and $q_0=p$ provides $u \in \left( H^{m+m_n+ s}_p(\Rn) \right)^l$.
Now we consider two different cases. 

In the case $\frac{m_n}{n} \geq \frac{1}{p}$ we know that $m+s-\frac{n}{q}<m+s+m_n-\frac{n}{p}$ holds for all
 $q \in [p, \infty)$. The embedding theorem for Bessel potential spaces, see e.g.\  \cite[Theorem 1.2]{Meyries},  therefore provides
  \begin{align*}
    u \in \left( H^{m+m_n+s}_p(\Rn) \right)^l \subseteq \left( H^{m+s}_q(\Rn) \right)^l \qquad \text{for all }  q \in [p, \infty). 
  \end{align*}

  In case $\frac{m_n}{n} < \frac{1}{p}$ we define $p_1 \in (p, \infty)$ by $\frac{1}{p_1} = \frac{1}{p}-\frac{m_n}{n}$. As in the first case we can apply the embedding theorem for Bessel potential spaces, see e.g. \cite[Theorem 6.5.1]{BerghLofstrom} and get that
  \begin{align*}
    u \in \left( H^{m+m_n+s}_p(\Rn) \right)^l \subseteq \left( H^{m+s}_{q}(\Rn) \right)^l 
  \end{align*}
for all $q \in [p,p_1]$.
  An application of \eqref{eq:FootStep} for $t=s$ and $q_0=p_1$ provides $u \in \left( H^{m+m_n+ s}_{p_1}(\Rn) \right)^l$. Since  $\frac{m_n}{n}>0$ is independent of $p$ we  obtain $\frac{m_n}{n} \geq \frac{1}{p_i}$, $i \in \N$ if we repeat the argument of this case finitely many times.  This implies 
  \begin{align}\label{eqFootStep2}
    u \in \left( H^{m+s}_{q}(\Rn) \right)^l  \qquad \text{for all $q \in [p,p_i]$.}
  \end{align}
Due to  $\frac{m_n}{n} \geq \frac{1}{p_i}$ we obtain by means of the embedding theorem for Bessel potential spaces, that $u \in  \left( H^{m+m_n+s}_{p_i}(\Rn) \right)^l \subseteq \left( H^{m+s}_{q}(\Rn) \right)^l $ for all $q \in [p_i, \infty)$. Together with \ref{eqFootStep2} we obtain the claim. 
\end{proof}

In order to prove the invariance of the Fredholm index of a non-smooth pseudodifferential operator under suitable conditions, we additionally need the next lemma, proved by Rabier in \cite[Lemma 3.4]{Rabier}:

\begin{lemma}\label{lemma:LemmaEstimateOfTheCodimension}
  Let $Y$ be a Banach space and $Z \subseteq Y$ be a closed subspace. We assume, that $D \subseteq Y$ is a dense subspace and that there are $d_1, \ldots, d_k \in D$ with the following property: For every $w \in D$,  there are scalars $\mu_1, \ldots, \mu_k$ such that $w-\sum_{i=1}^k \mu_i d_i \in Z$. Then we have $\codim Z \leq k$. 
\end{lemma}

With the previous results at hand, we are now in the position to prove the invariance result for the Fredholm index of non-smooth pseudodifferential operators, see Theorem \ref{thm:invarianceOfIndex}. The main idea is taken of \cite[Theorem 3.5]{Rabier}, where the invariance of the Fredholm index is proved for non-smooth differential operators.  
Since the Fredholm index is just defined for Fredholm operators, we first need to apply Theorem \ref{thm:Fredholmproperty}. On account of Theorem \ref{thm:Fredholmproperty} the Fredholm property of a suitable non-smooth pseudodifferential operator $a(x,D_x):H^{m+s}_p(\Rn)\rightarrow H^s_p(\Rn)$ is invariant with respect to $p \in [1,\infty)$ and $s \in \R$ fulfilling the assumptions of the continuity results for non-smooth pseudodifferential operators, cf. \cite[Theorem 2.7 and Theorem 4.2]{Marschall}. 
Then it remains to prove the invariance of the Fredholm index for such $p$ and $s$. Due to the regularity result of Lemma \ref{lemma:Lemma3.3}, we immediately obtain the invariance of the kernel of $a(x,D_x)$ with respect to $p$ and $s$. Additionally a combination of Lemma \ref{lemma:Lemma3.3} and the estimate of the codimension of the image of $a(x,D_x)$, cf. Lemma \ref{lemma:LemmaEstimateOfTheCodimension} provides the invariance of the codimension of $a(x,D_x):H^{m+s}_p(\Rn)\rightarrow H^s_p(\Rn)$ with respect to $p$ and $s$. 
A combination of those facts yields the claim of Theorem \ref{thm:invarianceOfIndex} for $l=1$. The other cases can be proved in the same way.

\begin{proof}[Proof of Theorem \ref{thm:invarianceOfIndex}:]
  On account of Theorem \ref{thm:Fredholmproperty} the operator $A^s_p$ is a Fredholm operator for all $p \in [1, \infty)$ and $p=2$ if $\rho \neq 1$ and $(1-\rho)\frac{n}{p}-(1-\delta)(\tilde{m}+\tau)<s<\tilde{m}+\tau$. Now let $p,q \in [1, \infty)$ with $p,q=2$ if $\rho \neq 1$ and $(1-\rho)\frac{n}{p}-(1-\delta)(\tilde{m}+\tau)<s,r<\tilde{m}+\tau$ be arbitrary. 
  It remains to show $\ind(A^s_p) = \ind(A^r_q)$. We immediately get by applying Lemma \ref{lemma:Lemma3.3} with $\Phi=0$ that $\ker(A^s_p)=\ker(A^r_q)$. Setting $k:= \codim \Im(A^s_p)$ we can choose $\Phi_1, \ldots, \Phi_k \in \left( H^s_p(\Rn) \right)^l$ such that $\Span \{ \Phi_1, \ldots, \Phi_k\} \oplus \Im(A^s_p)= \left( H^s_p(\Rn)\right)^l$. On account of the density of $C^{\infty}_c(\Rn)  \subseteq  H^s_p(\Rn) $ we even can assume that $\Phi_1, \ldots, \Phi_k \in \left( C_c^{\infty}(\Rn) \right)^l$. In order to get $\codim \Im(A^r_q) \leq k$, we just need to show that for all $\Phi \in \left( C^{\infty}_c(\Rn) \right)^l$ there are some $ \mu_1, \ldots, \mu_k \in \C $ such that
  \begin{align}\label{70e}
    \Phi- \sum_{j=1}^{k} \mu_j \Phi_j \in \Im(A^r_q) 
  \end{align}
  due to Lemma \ref{lemma:LemmaEstimateOfTheCodimension} applied on $X= \left( H^{m+r}_q(\Rn) \right)^l$, $Y= \left( H^r_q(\Rn) \right)^l$, $Z= \Im(A^r_q)$ and $D=\left( C_c^{\infty }(\Rn)\right)^l$.
  In order to verify (\ref{70e}) we take some fixed but arbitrary  $\Phi \in \left( C^{\infty}_c(\Rn) \right)^l$. Because of the choice of $\Phi_1, \ldots, \Phi_k$, there are some $ \mu_1, \ldots, \mu_k \in \C$ such that 
  \begin{align*}
    \Phi- \sum_{j=1}^{k} \mu_j \Phi_j \in \Im(A^s_p). 
  \end{align*}
  Hence there is a $u \in \left( H^{m+s}_p(\Rn) \right)^l$ with $A^s_p u =  \Phi- \sum_{j=1}^{k} \mu_j \Phi_j \in \left( C^{\infty}_c (\Rn)\right)^l$. Due to Lemma \ref{lemma:Lemma3.3} we therefore get $u \in \left( H^{m+r}_q(\Rn) \right)^l$, which provides (\ref{70e}). Changing the roles of $p$ and $q$ in the proof of $\codim \Im(A^r_q) \leq k$ provides $\codim \Im(A^r_q) \geq k$. A combination of $\ker(A^s_p) = \ker(A^r_q)$ and $\codim \Im(A^r_q) = \codim \Im(A^s_p)$ yields the claim.
\end{proof}


\section{Spectral Invariance of Non-Smooth Pseudodifferential Operators}\label{section:spectralInvariance}

 In this section we present the proof of the improved spectral invariance result for non-smooth pseudodifferential operators. 
An important role in this proof plays the following well known fact: The set of all invertible linear mappings is open. In combination with the spectral invariance result for smooth pseudodifferential operators, cf. Theorem \ref{thm:spectralInvarianceGlatterFall}, this enables us to check the following result for a given non-smooth pseudodifferential operator with symbol $a$, which is invertible, seen as a map form $H^{m+s}_p(\Rn)$ to $H^s_p(\Rn)$ for some suitable $p$ and $s$:
All smooth pseudodifferential operators with symbols belonging to a certain neighbourhood of $a$ are invertible as a map from $H^{m+r}_q(\Rn)$ to $H^r_q(\Rn)$ for all $q\in [1, \infty)$ and $r \in \R$. The matrix-valued case can be proved in the same way. 

For all $r \in \R$ and $1\leq q<\infty$ we write $B^r_q$ instead of a pseudodifferential operator $b(x,D_x)$, if we want to emphasize, that $b(x,D_x)$ is a map from $\left(H^{m+r}_q(\Rn)\right)^{l}$ to $\left(H^{r}_q(\Rn)\right)^{l}$. 
 
\begin{lemma}\label{lemma:GlatteInvertiblePDOinUmgebung}
  Let $\tilde{m} \in \N_0$, $0<\tau<1$, $m \in \R$ and $0 \leq \delta \leq  \rho \leq 1$ with $\rho>0$, $\delta <1$. Additionally let $a \in C^{\tilde{m}, \tau} S^{m}_{\rho, \delta}(\RnRn; \mathscr{L}(\C^l))$ and $b \in S^{m}_{\rho, \delta}(\RnRn;\mathscr{L}(\C^l))$.  We assume that $a(x,D_x):\left( H^{m+s}_p(\Rn) \right)^{l} \rightarrow \left(H^{s}_p(\Rn)\right)^{l}  $ is invertible for some $p \in [1, \infty)$ with $(1-\rho)\frac{n}{p}-(1-\delta)(\tilde{m}+\tau) < s < \tilde{m}+\tau$. In case $\rho \neq 1$ we assume $p=2$.
  Then there is some $k \in \N_0$ and some $R>0$ such that if $|a(x,\xi)-b(x,\xi)|^{(m)}_{k, C^{\tilde{m}, \tau} S^{m}_{\rho, \delta}(\RnRn; \mathscr{L}(\C^l))} < R$ the operator $B^r_q$ is invertible for every $q \in [1,\infty)$ and $r \in \R$, where $q=2$ if $\rho \neq 1$ is needed.
\end{lemma}

\begin{proof}
  Since the set of all invertible functions of $\mathscr{L}\left( (H_p^{s+m})^l, (H_p^{s})^l \right)$ is open, the invertibility of $a(x,D_x)$ provides the invertibility of all pseudodifferential operators with symbols $b \in S^{m}_{\rho, \delta}(\RnRn; \mathscr{L}(\C^l))$ if 
  \begin{align}\label{eq11}
    \| a(x,D_x) - b(x,D_x) \|_{\mathscr{L}\left( (H_p^{s+m})^l, (H_p^{s})^l \right)} \leq \e
  \end{align}
  for some $\e>0$.
   Now let  $(1-\rho)\frac{n}{p}-(1-\delta)(\tilde{m}+\tau) < s < \tilde{m}+\tau$ be arbitrary. On account of Theorem \ref{thm:BoundednessResultNonSmooth} there is a $k \in \N_0$ such that
  \begin{align*}
    \left\| \tilde{a}(x,D_x) \right\|_{\mathscr{L}\left( (H_p^{s+m})^l, (H_p^{s})^l \right)} \leq C_s |\tilde{a}|^{(m)}_{k, C^{\tilde{m}, \tau} S^m_{\rho,\delta}(\RnRn; \mathscr{L}(\C^l))} 
  \end{align*}
  for all $\tilde{a} \in  C^{\tilde{m}, \tau} S^m_{\rho,\delta}(\RnRn; \mathscr{L}(\C^l))$. 
  Consequently there is a $R>0$ such that for all $b \in S^{m}_{\rho, \delta}(\RnRn; \mathscr{L}(\C^l))$ with
  \begin{align}\label{eq12}
    |a(x,\xi)-b(x,\xi)|^{(m)}_{k, C^{\tilde{m}, \tau} S^{m}_{\rho, \delta}(\RnRn; \mathscr{L}(\C^l))} < R
  \end{align}
 inequality (\ref{eq11}) holds. For the symbols  $b \in S^{m}_{\rho, \delta}(\RnRn; \mathscr{L}(\C^l))$ fulfilling (\ref{eq12}) the corresponding pseudodifferential operator
 $b(x,D_x)$ is invertible in $\mathscr{L}\left( (H_p^{s+m})^l, (H_p^{s})^l \right)$. Due to Theorem \ref{thm:spectralInvarianceGlatterFall}, $B_q^r$ is invertible for all $r \in \R$ and all $1\leq q<\infty$, where $q=2$ if $\rho \neq 1$.
\end{proof}

With all results of the previous sections we are now in the position to show Theorem \ref{thm:SpectralInvarianceNonSmoothCase}.
We follow the proof of Theorem 4.4 in \cite{Rabier} and prove this theorem by contradiction. Consequently we  assume that $A^r_{q_0}$ is not invertible for some $q\in [1, \infty)$.  According to the next lemma, which was proved in \cite[Lemma 4.3]{Rabier}, we cannot find a neighbourhood of $A^r_{q_0}$ in $\mathcal{L}(H^{m+r}_{q_0}(\Rn), H^r_{q_0}(\Rn))$, in which all operators are invertible. 
However Lemma \ref{lemma:GlatteInvertiblePDOinUmgebung} and the result about the invariance of the Fredholm index, cf. Theorem \ref{thm:invarianceOfIndex} lead to a contradiction to that conclusion.

\begin{lemma}\label{lemma:Rabier}
  Let $X$ and $Y$ be complex Banach spaces and let $T \in \mathscr{L}(X, Y)$ be  Fredholm of index $0$ and not invertible. Then there is an open ball $B_{\nu}(0) \subseteq \mathscr{L}(X, Y)$ with the following property: Given $H \in B_{\nu}(0)$ such that $T+H$ is invertible and $\e >0$, there is some $\delta \in (0, \e]$ such that if $S \in B_{\delta}(T) \subseteq \mathscr{L}(X, Y)$, then $S+zH$ is not invertible for some $z \in \C$ with $|z| < \e$.
\end{lemma}

\begin{proof}[Proof of Theorem \ref{thm:SpectralInvarianceNonSmoothCase}]
    We assume that $A^r_{q_0}$ is not invertible for some ${q_0} \in [1, \infty)$ and $r \in \R$ with $(1-\rho)\frac{n}{p}-(1-\delta)(\tilde{m}+\tau) < r < \tilde{m}+\tau$. Since $A^s_p$ is invertible, $A^s_p$ is a Fredholm operator with index $0$. Due to Theorem \ref{thm:invarianceOfIndex} we know, that $A^r_{q_0}$ is also a Fredholm operator with index $0$. Moreover we choose $R>0$ and $k \in \N_0$ as in Lemma \ref{lemma:GlatteInvertiblePDOinUmgebung} and $\nu>0$ as in Lemma \ref{lemma:Rabier} with $X:= \left( H_{q_0}^{m+r} \right)^l$, $Y:= \left( H_{q_0}^{r} \right)^l$ and $T:=A^r_{q_0}$. By means of the proof of Lemma \ref{lemma:GlatteInvertiblePDOinUmgebung} and  Lemma \ref{DichtheitGlatteSymbolklasseInNonSmoothOne} in case $\delta=0$ and Lemma  \ref{lemma:KonvergenzFürDeltaUngleich0} else there is a symbol $b \in S^m_{\rho,\delta}(\RnRn; \mathscr{L}(\C^l))$ such that 
$$|a-b|^{(m)}_{k, C^{\lfloor t \rfloor, t- \lfloor t \rfloor} S^{m}_{\rho, \delta}(\RnRn; \mathscr{L}(\C^l))} < R$$ 
with 
\begin{itemize}
	\item $t:=\max\{ r+(\tilde{m}+\tau -r)/2, s+(\tilde{m}+\tau -s)/2 \}$ and 
	\item $ \left\| b(x,D_x)-a(x,D_x)\right\|_{\mathscr{L}\left( (H_{q_0}^{m+r})^l; (H_{q_0}^r)^l \right)} < \nu$.
\end{itemize}
 An application of Lemma \ref{lemma:GlatteInvertiblePDOinUmgebung} provides the invertibility of
  \begin{align*}
    b(x,D_x) = a(x, D_x) + \left( b(x,D_x) - a(x,D_x) \right) \in \mathscr{L}\left( (H_{q_0}^{m+r})^l; (H_{q_0}^r)^l \right).
  \end{align*}
  Hence $a(x, D_x) + \tilde{a}(x, D_x) \in \mathscr{L}\left( (H_{q_0}^{m+r})^l; (H_{q_0}^r)^l \right)$ is invertible if 
  \begin{align}\label{eq13}
    \| \tilde{a}(x, D_x) - \left( b(x,D_x) - a(x,D_x) \right) \|_{\mathscr{L}\left( (H_{q_0}^{m+r})^l; (H_{q_0}^r)^l \right)}
  \end{align}
  is small enough. Because of Lemma \ref{DichtheitGlatteSymbolklasseInNonSmoothOne} in case $\delta=0$ and Lemma \ref{lemma:KonvergenzFürDeltaUngleich0} else, there is a symbol $\tilde{a} \in S^m_{\rho,\delta}(\RnRn; \mathscr{L}(\C^l))$ such that
  \begin{align*}
    &\| \tilde{a}(x, D_x) - \left( b(x,D_x) - a(x,D_x) \right) \|_{\mathscr{L}\left( (H_{q_0}^{m+r})^l; (H_{q_0}^r)^l \right)} \\
    &\qquad \qquad \qquad <  \nu - \| b(x, D_x) - a(x, D_x) \|_{  \mathscr{L}\left( (H_{q_0}^{m+r})^l; (H_{q_0}^r)^l \right)}
  \end{align*}
  is so small that $a(x, D_x) + \tilde{a}(x, D_x) \in \mathscr{L}\left( (H_{q_0}^{m+r})^l; (H_{q_0}^r)^l \right)$ is invertible. Consequently
  \begin{align*}
    \| \tilde{a}(x, D_x) \|_{  \mathscr{L}\left( (H_{q_0}^{m+r})^l; (H_{q_0}^r)^l \right)}
    < \nu.
  \end{align*}
  Now let us choose $\e:= \frac{R}{2 \max\{1, \left| \tilde{a}(x, D_x) \right|^{(m)}_{k, C^{\lfloor t \rfloor, t- \lfloor t \rfloor} S^{m}_{\rho, \delta}(\RnRn; \mathscr{L}(\C^l))}\}}$.  Lemma \ref{lemma:Rabier} provides the existence of a $\tilde{\delta} \in (0, \e]$ with the following property: For each $\tilde{b}(x, D_x) \in \op S^m_{\rho,\delta}(\RnRn; \mathscr{L}(\C^l))$ satisfying
  \begin{align}\label{eq20}
    \|\tilde{b}(x, D_x)-a(x, D_x)\|_{ \mathscr{L}\left( (H_{q_0}^{m+r})^l; (H_{q_0}^r)^l \right)} < \tilde{\delta}
  \end{align}
  and for some $z \in \C$ with $|z| <\e$ the operator 
  \begin{align}\label{eq21}
    \tilde{b}(x, D_x)+z\tilde{a}(x, D_x) \qquad \text{ is not invertible.}
  \end{align}
 Due to Lemma \ref{DichtheitGlatteSymbolklasseInNonSmoothOne} in  case $\delta=0$ and Lemma \ref{lemma:KonvergenzFürDeltaUngleich0} else there is a $\tilde{b}(x, D_x) \in \op  S^m_{\rho,\delta}(\RnRn; \mathscr{L}(\C^l))$ such that (\ref{eq20}) holds and in addition $| \tilde{b}-a |^{(m)}_{k, C^{\lfloor t \rfloor, t- \lfloor t \rfloor} S^{m}_{\rho, \delta}(\RnRn; \mathscr{L}(\C^l))} <R/2$. Hence
  $| \tilde{b} + za-a |^{(m)}_{k, C^{\lfloor t \rfloor, t- \lfloor t \rfloor} S^{m}_{\rho, \delta}(\RnRn; \mathscr{L}(\C^l))} <R$. An application of Lemma \ref{lemma:GlatteInvertiblePDOinUmgebung} yields the invertibility of the operator $\tilde{b}(x, D_x) + z \tilde{a}(x, D_x) \in  \mathscr{L}\left( (H_{q_0}^{m+r})^l; (H_{q_0}^r)^l \right)$. Since this is a contradiction to (\ref{eq21}), we conclude the proof. 
\end{proof}

As a direct consequence of the Theorem \ref{thm:SpectralInvarianceNonSmoothCase} we obtain

\begin{kor}
  Let $\tilde{m} \in \N$, $m \in \R$,  $0 \leq \delta < \rho \leq 1$, $0<\tau<1$ with $\tilde{m}+\tau> \frac{1-\rho}{1-\delta}\cdot \frac{n}{2}$ if $\rho \neq 1$ and $q \in [1, \infty)$. Moreover let $a \in C^{\tilde{m}, \tau} \tilde{S}^{m}_{\rho, \delta}(\RnRn; \mathscr{L}(\C^l))$ for $ \delta=0$ and $a \in C_{unif}^{\tilde{m},\tau} S^{m}_{\rho, \delta}(\RnRn; \mathscr{L}(\C^l)) \cap C^{\tilde{m}, \tau} \tilde{S}^{m}_{\rho, \delta}(\RnRn; \mathscr{L}(\C^l))$ else fulfilling the assumptions 1) and 2) of Theorem \ref{thm:invarianceOfIndex}. Then the spectrum of $a(x, D_x)$ seen as an operator from $\left( H_q^{m+r}\right)^l$ to $\left( H_q^{r}\right)^l$, where $q=2$ if $\rho \neq 1$, is independent of $r\in \R$ with $-(1-\delta)(\tilde{m}+\tau) < r< \tilde{m}+\tau$. In case $\rho =1$ the spectrum is even independent of $q \in [1, \infty)$.
\end{kor}


\end{document}